\documentclass[preprint,12pt]{elsarticle}

\usepackage{graphics}

\usepackage{amssymb}
\usepackage{lineno}
\usepackage{amsthm}
\usepackage{amsmath}
\usepackage{mathabx}
\usepackage{color}
\usepackage{amsopn}
\usepackage{amsmath}
\usepackage{mathabx}
\usepackage{pstricks}
\usepackage{color}
\usepackage{pst-plot}
\usepackage{pst-tree}
\usepackage{pst-eps}
\usepackage{multido}
\usepackage{pst-node}
\usepackage{pst-3d}
\usepackage{colordvi}
\usepackage{multirow}
\usepackage{pst-grad}
\usepackage{colortbl}

\DeclareMathOperator{\interior}{Int}

\newtheorem{theorem}{Theorem}[section]

\newtheorem{corollary}[theorem]{Corollary}
\newtheorem{definition}[theorem]{Definition}

\newtheorem{lemma}[theorem]{Lemma}
\newtheorem{proposition}[theorem]{Proposition}

\journal{ }

\begin{document}

\renewcommand{\listtablename}{\'indice de tablas}

\nocite{*}

\begin{frontmatter}



\title{Minimax Hausdorff estimation of density level sets}

\author[rvt]{A. Rodr\'iguez-Casal}
\author[rvt]{P. Saavedra-Nieves \corref{cor1}}
\address[rvt]{Department of Statistics, Mathematical Analysis and Optimization, University of Santiago de Compostela, Spain}
\cortext[cor1]{Corresponding author: paula.saavedra@usc.es (P. Saavedra-Nieves)}

\begin{abstract}
Given a random sample of points from some unknown density, we propose a data-driven method for estimating density level sets under the $r-$convexity assumption. This shape condition generalizes the convexity property. However, the main problem in practice is that $r$ is an unknown geometric characteristic of the set related to its curvature. A stochastic algorithm is proposed for selecting its optimal value from the data. The resulting reconstruction of the level set is able to achieve minimax rates for Hausdorff metric and distance in measure, up to log factors, uniformly on the level of the set.  
\end{abstract}

\begin{keyword}Nonparametric density level set estimation \sep $r-$convexity \sep smoothing parameter
\end{keyword}

\end{frontmatter}



\section{Introduction}\label{1}Given a random sample of points $\mathcal{X}_n$ from a density $f$, level set estimation theory deals with the problem of reconstructing, for a given level $t>0$, the unknown set  \begin{equation}\label{conjuntonivel1}G(t)=\{x\in\mathbb{R}^d:f(x)\geq t\}.\end{equation}

Many applications have emerged  in different scientific fields since the concept of population clusters was introduced. Clusters are defined as the connected components of density level sets, see \cite{har}, \cite{cuefra} or \cite{rin}. This definition depends on the user-specified level. For adressing this issue, an algorithm for estimating the smallest level at which there are more than one connected components in \cite{ste}. Furthermore, the notions of clustering and mode are intimately related, see \cite{cue}. An interesting application of this clustering approach to the astronomy was proposed in \cite{jan}. From a similar point of view, methods for visualizing multivariate density estimates were introduced in \cite{kle1} and \cite{kle2}. In addition, level set estimation was used in the context of the Hough transform, see \cite{gol}. In \cite{roe}, level sets were reconstructed for analyzing the difference between two probability densities in the field of flow cytometry. The detection of mine fields based on aerial observations, the analysis of seismic data, as well as certain issues in image segmentation involve level set estimation, see \cite{huo}. The detection of outliers is another key application, see \cite{gar} or \cite{mar} for a review. An outlier can be thought of as an observation that does not belong to the \textit{effective support} of the distribution. This one can be represented using a level set that looks like the support without including those areas that are almost empty from the probabilistic point of view. This scheme follows that of \cite{dev} to determine whether a manufacturing process is out of control. For more quality control approaches, see \cite{bai1} and \cite{bai2}.\vspace{-.8cm}	
\begin{figure}[h!]      		 	  \hspace{2cm}\includegraphics[scale=.26]{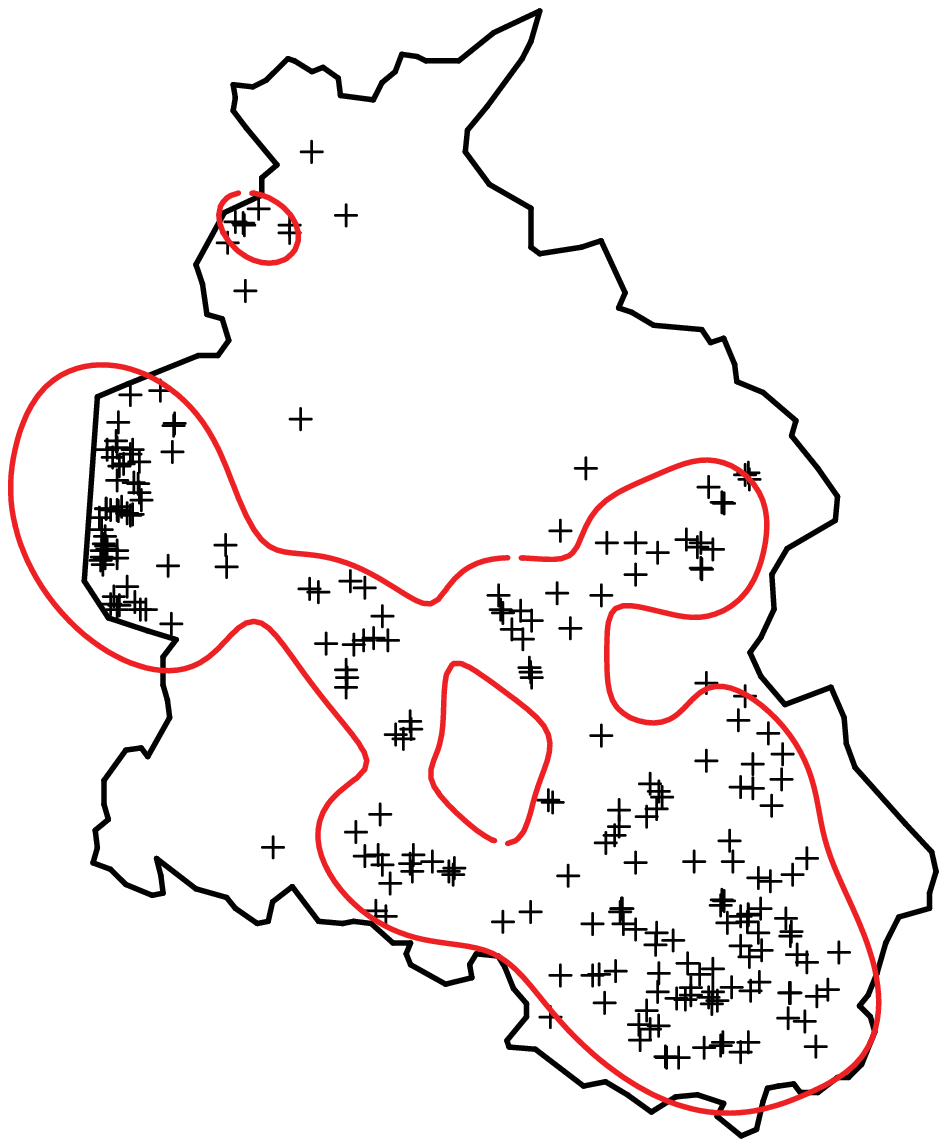}\hspace{-1cm}\includegraphics[scale=.26]{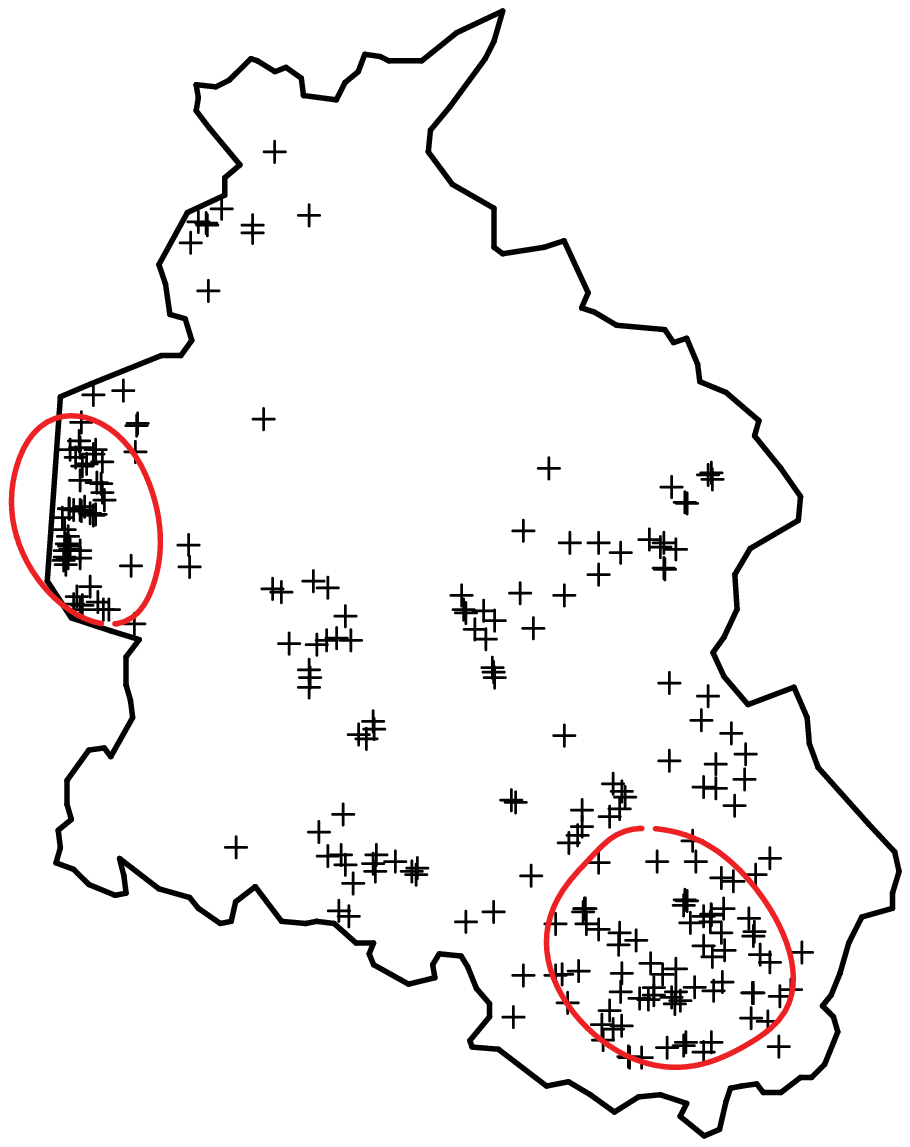}\hspace{-1cm}\includegraphics[scale=.26]{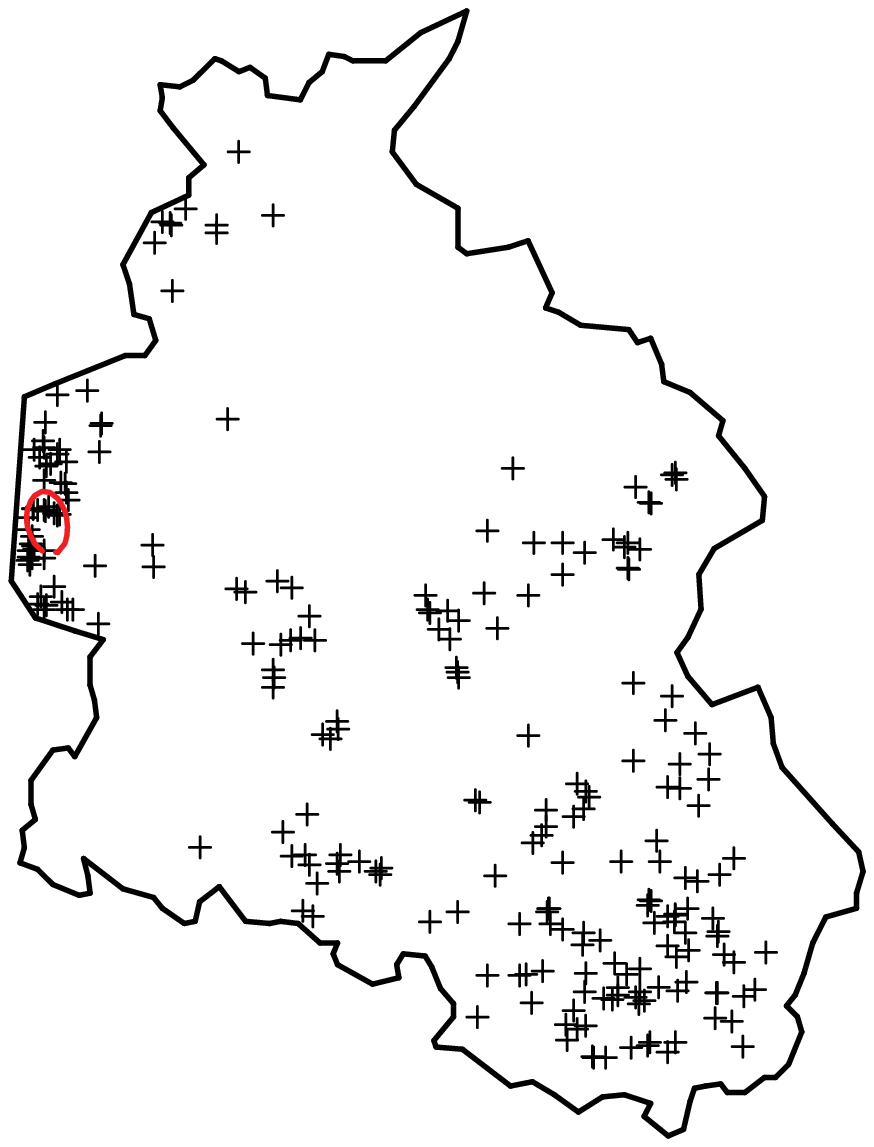} \vspace{-1.3cm}
	\caption{Level sets for the distribution of 322 cases diagnosed of leukaemia on the North West of England with $\tau=0.05$ (right), $\tau=0.5$ (center) and $\tau=0.95$ (left).}\vspace{-.2cm}
	\label{leucemia1}
\end{figure}      		

Just as quality control case, the areas of the distribution support
where $f$ is close to zero are usually of lesser interest for many practical
purposes since the probability of finding points there is extremely low. However, the value of level $t$ in (\ref{conjuntonivel1}) is unknown if the level set $G(t)$ must satisfy a fixed probability content. Therefore, an alternative definition for density level sets is considered next. Given $\tau\in (0,1)$, it is defined
\begin{equation}\label{conjuntonivel2}
	L(\tau)=\{x\in\mathbb{R}^d:f(x)\geq f_\tau\}
\end{equation}where
\begin{equation}\label{umbral}
	f_\tau=\sup
	\left\{y\in(0,\infty):\int_{\mathbb{R}^d}f(t)\mathbb{I}_{\{f(t)\geq
		y\}}\geq1-\tau\right\}.
\end{equation}
Therefore, $L(\tau)$ can be written as $G(f_\tau)$. According to equation (\ref{umbral}), $f_\tau$ denotes
the largest threshold such that the level set $L(\tau)$ has a probability greater than or equal to $1-\tau$ with
respect to the distribution induced by $f$. Therefore, for values of $\tau$ close to one, the level
set $L(\tau)$ represents the domain concentrated around the greatest mode. However, if $\tau$
is close to zero then it represents the \textit{substancial support} of the density $f$.

Figure \ref{leucemia1} contains the residential coordinates for 322 cases diagnosed of chronic granulocytic leukemia in the North West of England between 1982 up to 1998 (inclusive), see \cite{hen} and \cite{dig} for details on data set. In addition, contours of level sets have been represented for different values of $\tau$. They can be used for analyzing the existence of spatial clustering of leukaemia.

There exist three different nonparametric methodologies for estimating density levels sets in literature: plug-in, excess mass and hybrid methods. For a Bayesian point of view, see \cite{gay}. Next, they are detailed briefly:

The \emph{plug-in estimation} is the most natural choice to estimate $L(\tau)$ when no geometric information about the level set is available. It based on replacing $f$ by a nonparametric estimator for the density $f_n\label{fn}$ in (\ref{conjuntonivel2}). Given $\mathcal{X}_n=\{X_1,\cdots,X_n\}$, the kernel density estimator at point $x$ is defined as\vspace{.1cm}
\begin{equation}
	\label{estimacionnucleo}f_n(x)=\frac{1}{nh^d}\sum_{i=1}^n K\left(\frac{x-X_i}{h}\right),
\end{equation}where $h  $ is a bandwidth and $K$, a kernel function. The estimator defined in (\ref{estimacionnucleo}) is heavily dependent on $h$, see \cite{chacon}. Thus, this group of methods
proposes $\hat{L}(\tau) = \{f_n\geq \hat{f}_\tau\}$ as an estimator of $L(\tau )$, where $\hat{f}_\tau$ denotes an estimator of the threshold $f_\tau$, see \cite{cad1}, \cite{hyn} and \cite{cad2}. The plug-in methodology is the most common approach, and has received
considerable attention in the literature, e.g., \cite{r4}, \cite{bai3}, \cite{mas}, \cite{rig}, \cite{rig2}, \cite{pol2} or \cite{chen}. However, one practical problem of the plug-in methodology is the choice of matrix $h$. Unlike density estimation, the level set estimation has been considered in literature from many points of view but, in general, without deepening in methods for selecting $h$. In fact, this problem was first considered by \cite{bai2} in the context of nonparametric statistical quality control. A plug-in procedure that is based on an empirical density estimator, the regular histogram, was presented in \cite{sin}. Later, an automatic bandwidth selection rule to estimate density level sets but only in the one-dimensional case was derived in \cite{sam}.

The \emph{excess mass estimation} assumes that the researcher has information a priori about the shape of the level set $G(t)$. This approach was first proposed by \cite{har2} and \cite{mul}. See \cite{gru} too. Some previous contributions can be seen in \cite{cher} and \cite{edd}. Then, \cite{pol} extended and investigated it in a very general framework and \cite{mul1} proposed an efficient algorithm for estimating one-dimensional
sets by assuming that the theoretical level set can be written as a finite union
of $M$ closed intervals assuming that $M$ is known a priori. These algorithms assume that $G(t)$ maximizes the functional\vspace{-.1mm}
$$H_t(B)=\mathbb{P}(B)-t\mu(B),\vspace{-.1mm}$$on the Borel sets $B$ where $\mathbb{P}$ denotes the probability measure induced by $f$ and $\mu\label{lebesgue2}$, the Lebesgue measure. Furthermore, $H_t$ can be estimated empirically. So, if $G(t)$ is assumed to belong to a family of sets then it could be reconstructed by maximizing the empirical version of the previous functional on the family considered. Consequently, unlike the plug-in approximation, excess mass methods do not need to smooth the sample $\mathcal{X}_n$ and, in addition, they impose geometric restrictions on the estimators. These methods were not designed for estimating the level set $L(\tau)$ but they can be adapted easily, see \cite{saa}.

The last methodology is a \emph{hybrid} of the two previous ones. Just as the excess mass methods, the hybrid methodology assumes some shape restrictions on the class of sets considered and, like the plug-in methods, it needs to smooth the data set. In \cite{r2}, it is proposed the granulometric smoothing method to reconstruct level sets $L(\tau)$ assuming that $L(\tau)$ and the closure of its complement $\overline{L(\tau)^c}$ are both $r-$convex.

This latter shape restriction generalizes the convexity property, see \cite{r33}. A closed set $A\subset\mathbb{R}^d$ is said to be $r-$convex, for some $r>0$, if $A=C_{r}(A)$, where
$$C_{r}(A)=\bigcap_{\{B_r(x):B_r(x)\cap
	A=\emptyset\}}\left(B_r(x)\right)^c$$
denotes the $r-$convex hull of $A$ and $B_r(x)$, the open ball with
center $x$ and radius $r$. In Figure \ref{coroa}, $C_{r}(A)$ is shown for different values of $r$ when $A$ is a uniform sample on a circular ring. Note that the boundary of the $r-$convex hull is formed by arcs of balls of radius $r$ (besides possible isolated sample points). If $A$ is $r-$convex, it is easy to prove that $A$ is also $r^*-$convex for all $0<r^*\leq r$. Furthermore, it can be seen that $C_{r^*}(A)\subset C_{r}(A)$ for all $0<r^*\leq r$. See \cite{r2} and \cite{kor22} for more details on $r-$convexity.

In addition, if $A$ and $\overline{A^c}$ are both $r-$convex, it is verified that the reach of the boundary of $A$ is bigger than $r$, see \cite{arias}. The reach $r_\mathfrak{M}$ of a manifold $\mathfrak{M}$ is an important geometric characteristic that corresponds to the largest number such that any point at distance less than $r_\mathfrak{M}$ from $\mathfrak{M}$ has an unique nearest
point on $\mathfrak{M}$. Its estimation plays an interesting role in manifold reconstruction, homological inference, volume estimation or manifold clustering. See \cite{aamari} for motivation and references.

The $r-$convex hull is also closely related to the closing of $A$ by $B_r(0)$ from the mathematical morphology, see \cite{serr}. It can be shown that
$$C_{r}(A)=(A\oplus r B)\ominus r B,$$
where $B=B_1(0)$, $\lambda C=\{\lambda c: c\in C\}$, $C\oplus D=\{c+d:\ c\in C, d\in D\}$ and $C\ominus D=\{x\in\mathbb{R}^d:\ \{x\}\oplus D\subset C\}$, for $\lambda \in \mathbb{R}$ and
sets $C$ and $D$.\vspace{-.15cm}

\begin{figure}[h!]  	 	
	\hspace{1.5cm}\includegraphics[scale=.24]{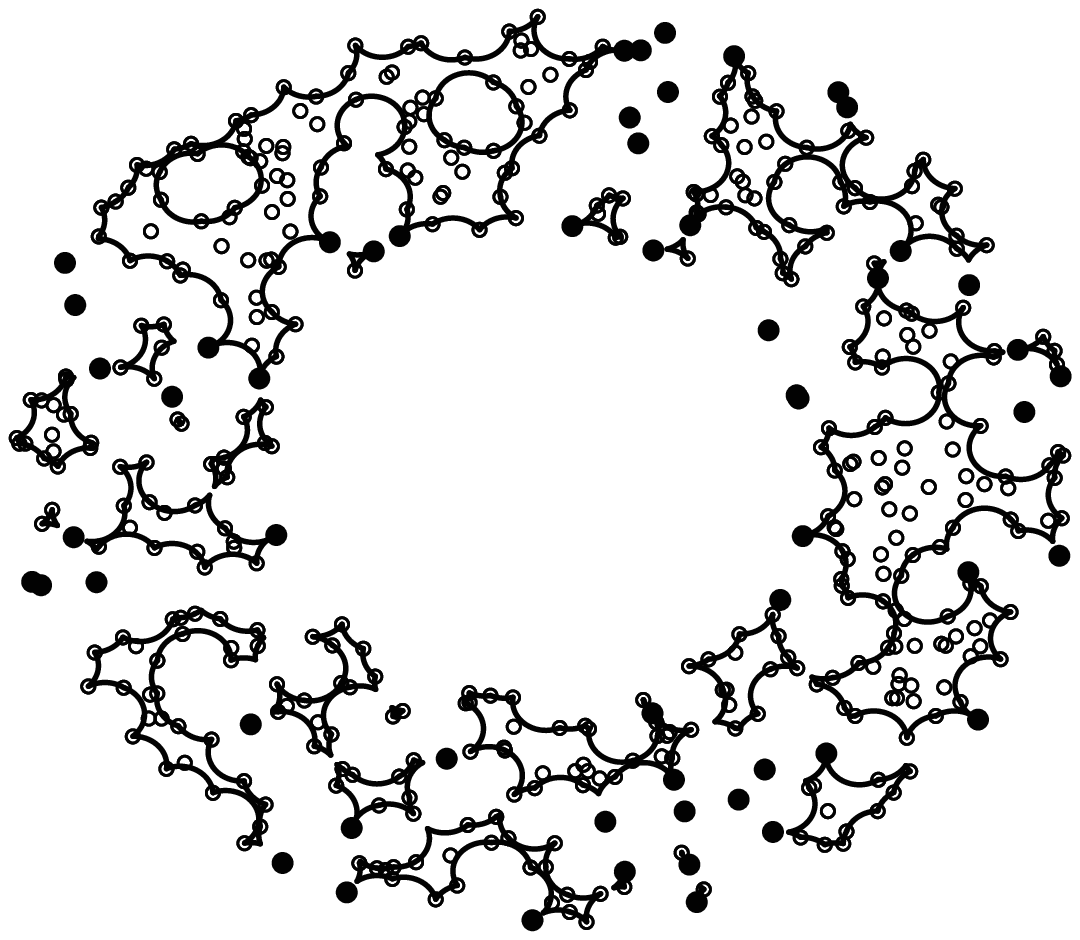}\includegraphics[scale=.24]{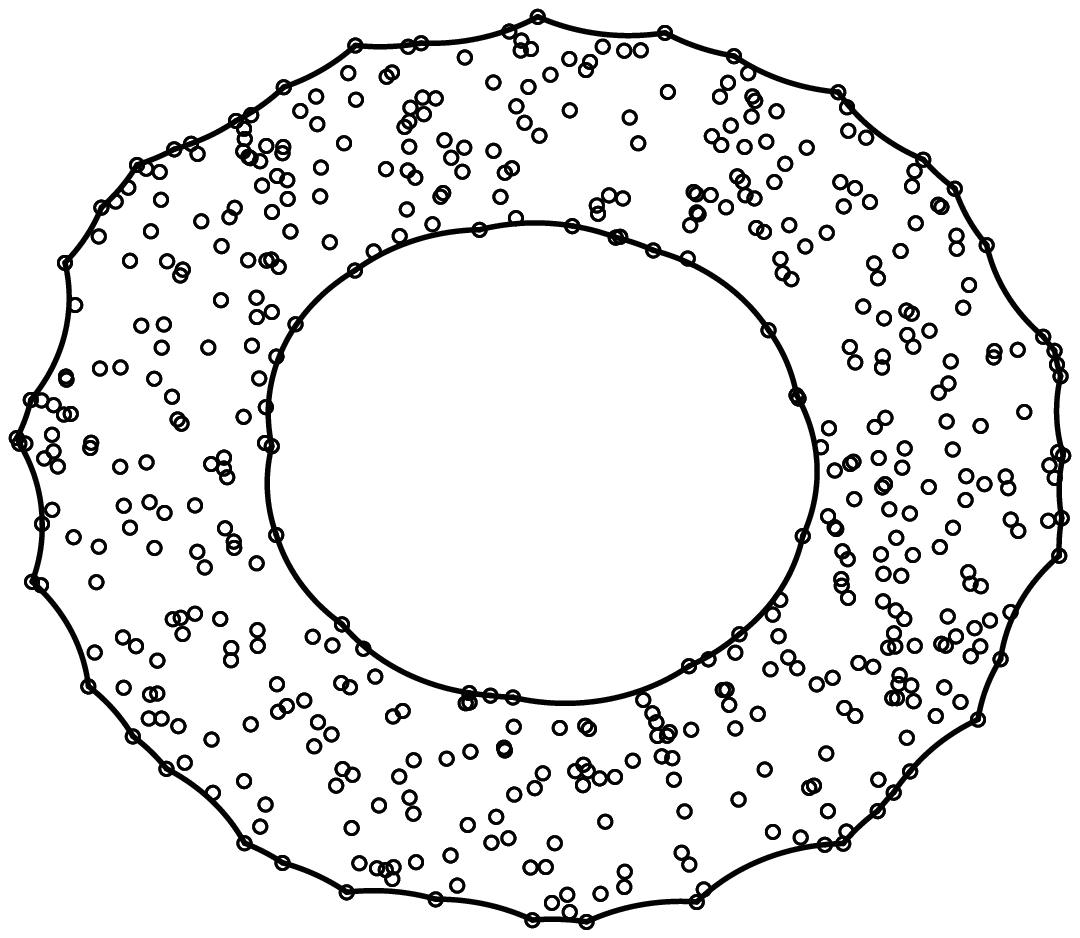}\includegraphics[scale=.24]{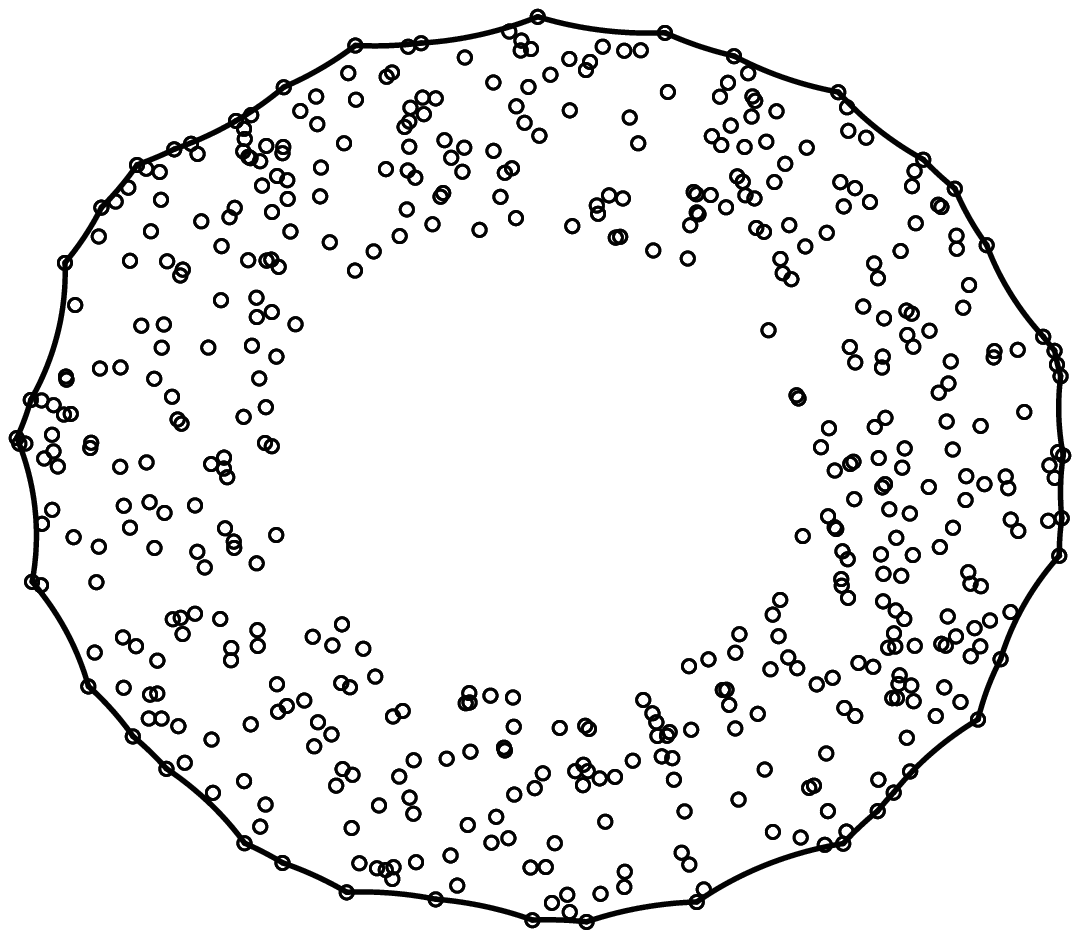}\\
	\vspace{-1.05cm}
	\caption{$A$ denotes a uniform sample on the circular ring $B_{0.4}[(0,0)]\setminus B_{0.2}[(0,0)] $. $C_r(A)$ is shown for $r$ equal to $0.025$ (right), $0.15$ (center) and $0.25$ (left).}\vspace{-.15cm}
	\label{coroa}
\end{figure}

If a density level set is assumed to be $r-$convex and the subset $\mathcal{X}_n^f$ of $\mathcal{X}_n$ inside the level set was known, the most natural estimator would be the $r-$convex hull of $\mathcal{X}_n^f$. However, in practice, the set $\mathcal{X}_n^f$ is unknown and available results in literature do not give any criterion for selecting the smoothing parameter $r$ from $\mathcal{X}_n$. The aim of this paper is to overcome these two drawbacks for proposing a fully data-driven level set estimator. A consistent estimator of the parameter $r$ will be proposed and a subset of $\mathcal{X}_n$ contained, with probability one, in the level set will be determinated. We will show that the proposed method is optimal in the minimax sense.

Two metrics between sets are considered in order to assess the performance of a level set estimator. Let $A$
and $C$ be two closed, bounded, nonempty subsets of $\mathbb{R}^{d}$. The
Hausdorff distance between $A$ and $C$ is defined by\vspace{-0.13cm}
$$
d_{H}(A,C)=\max\left\{\sup_{a\in A}d(a,C),\sup_{c\in C}d(c,A)\right\},\vspace{-0.13cm}
$$
where $d(a,C)=\inf\{\|a-c\|:c\in C\}$ and $\|\mbox{ }\|$ denotes the Euclidean norm.
On the other hand, if $A$ and $C$
are two bounded and Borel sets then the distance in measure between $A$ and $C$ is defined by $d_{\mu}(A,C)=\mu(A\triangle C)$, where $\triangle$ denotes the symmetric difference, that is, $A\triangle C=(A \setminus C)\cup(C \setminus A). $

This paper is organized as follows. In Section \ref{2}, the estimator of the parameter $r$ is formally defined. Uniform consistency on level $t$ for the parameter estimator is established in Section \ref{3}. In addition, it is shown that the resulting level set estimator is able to achieve minimax rates for Hausdorff metric and distance in measure, up to log factors, in most cases where the minimax rates are known, see \cite{mam}. Unlike method proposed in \cite{r2}, these rates do not depend on any penalty term where $r$ is not estimated from $\mathcal{X}_n$. The numerical questions involving the practical application of the algorithm are analyzed in Section \ref{4}. Furthermore, an illustration using a real data set is shown. In Section \ref{5}, conclusions were established. Finally, proofs are deferred to Section \ref{6}.

Supplementary material is contained in two appendices. In Appendix A, some theoretical results in \cite{r2} are summarized. They will be used in many proofs. Appendix B contains the proofs of some auxiliary results that will allow to get general consistency results.

\section{Selection of the optimal smoothing parameter}\label{2}	

The problem of reconstructing a $r-$convex density level set using a data-driven procedure can be solved if the smoothing parameter $r$ is estimated from the random sample $\mathcal{X}_n$. The first step is to determine precisely the optimal value of $r$ to be estimated. We propose to estimate the highest value of $r$ which verifies that $G(t)$ is $r-$convex. Of course, this value depends on the level $t$

\begin{definition}\label{r_0_t}Let $G(t)$ be a compact, nonempty, nonconvex and $r-$convex level set for some $r>0\label{r0t}$. It is defined
	\begin{equation}\label{estimadorr0levelsetu}
		r_0(t)=\sup\{\gamma>0:C_\gamma(G(t))=G(t)\}.
	\end{equation}
\end{definition}For simplicity in the exposition, it is assumed that $G(t)$ is not convex. Of course, if $G(t)$ is convex, $ r_0(t)$ would be infinity. Therefore, the convex hull could be used to estimate $G(t)$ intead of the $r-$convex hull. In addition, it can be proved that, under a mild regularity condition, the supreme established in (\ref{estimadorr0levelsetu}) is a maximum, that is, $G(t)$ is also $r_0(t)-$convex. The regularity property we need to establish the consistency results is slightly stronger than $r-$convexity:\vspace{4mm}\\
($R_{\lambda}^r$) A closed ball of radius $\lambda>0$ rolls freely in $G(t)$ (inside) and a closed \\ \textcolor[rgb]{1.00,1.00,1.00}{$ $ $ $ $ $    $ $ $ $ $ $ } ball of radius $r>0$ rolls freely in $\overline{G(t)^c}$ (outside).\vspace{.3mm}\\

Following \cite{r2} and \cite{r33}, a closed ball of radius $\gamma>0$ rolls freely in a closed set $A$ if for each boundary point $b\in\partial A$ there exists $x\in\mathbb{R}^d$ such that $b\in B_\gamma[x]\subset A$ where $B_\gamma[x]$ denotes the closed ball of radius $\gamma$ centering at $x$. Under ($R_{\lambda}^r$), we can justify the optimality of $r_0(t)$, see  (\ref{estimadorr0levelsetu}). It is clear that $G(t)$ is $r-$convex for $r \leq r_0(t)$ but if $r < r_0(t)$, $C_r(\mathcal{X}_n^f)$ is
a non admisible estimator since it is always outperformed by $C_{r_0(t)}(\mathcal{X}_n^f)$. Then, if $r \leq r_0(t)$ it is verified that, with probability one,
$C_r(\mathcal{X}_n^f) \subset C_{r_0 (t)}(\mathcal{X}_n^f) \subset G(t)$ and hence, $d_\mu(C_{r_0(t)} (\mathcal{X}_n^f), G(t)) \leq d_\mu(C_r(\mathcal{X}_n^f), G(t))$ (the same holds for the Hausdorff distance).
It should also noted that, for $r > r_0(t)$, $C_r(\mathcal{X}_n^f)$ would considerably overestimate $G(t)$.

In addition, satisfying the shape condition ($R_{\lambda}^r$) is a quite natural general property for level sets of densities. In fact, Theorem 2 in \cite{r2} proved that, under some assumptions on the density $f$, its level sets satisfy ($R_{\lambda}^r$) for $r=\lambda=m/k$. Then, according to Theorem 2 in \cite{r2}, the following assumptions are considered on $f$:

\begin{description}
	\item[A. ]\begin{enumerate}\item The threshold $t$ of $G(t)$ belongs to $[l,u]$ with $-\infty<l\leq u<\sup (f)<\infty$.
		\item $f\in\mathcal{C}^p(U)$, $p\geq 1$ where $U$ is a bounded open set containing $\overline{G(l-\zeta)}\setminus \interior (G(u+\zeta))$ for some $\zeta>0$  where $G(u+\zeta)$ is bounded.
		\item The gradient of $f$, $\nabla f$, satisfies $|\nabla f|\geq  m  >0$ as well as Lipschitz condition on $U$:
		$$|\nabla f(x)-\nabla f(y)|\leq k|x-y|\mbox{ for }x,y\in U.$$
	\end{enumerate}
\end{description}

Under (A), it is verified that $r_0(t)\geq m/k$. In addition, sets fulfilling condition ($R_{\lambda}^r$) have a number of desirable properties which make them easier to handle. In particular, it can be seen that condition ($R_{\lambda}^r$) implies $r-$convexity. More generally, it can be proved that,
under the condition ($R_{\lambda}^r$), the
$r$-rolling condition implies the
$r-$convexity. It
should be noted that this is not true in general, see \cite{cue22}, but
this implication holds for the smooth sets which satisfy ($R_{\lambda}^r$).

The estimator for the optimal parameter defined in (\ref{estimadorr0levelsetu}) depends  on a sequence $D_n$ satisfying the assumption:
\begin{description}
	\item[D. ]$D_n$ is equal to $M(\log{n}/n)^{p/(d+2p)}$ for a big enough value of the constant $M>0$.\label{Dn}
\end{description}

\begin{definition}\label{jejejeje}Let $G(t)$ be a compact, nonempty and nonconvex level set. Under assumptions (A) and (D), let $\mathcal{X}_n$ be a random sample generated from a distribution with density function $f$. An estimator for the parameter established in Definition \ref{r_0_t} can be defined as
	\begin{equation}\label{estimadorr0hat}
		\hat{r}_0(t)=\sup\{\gamma>0:C_\gamma(\mathcal{X}_n^+(t))\cap \mathcal{X}_n^-(t)=\emptyset\},
	\end{equation}where
	$$
	\mathcal{X}_n^+(t)=\{X\in\mathcal{X}_n: f_n(X)\geq t+D_n\}\mbox{ and }\mathcal{X}_n^-(t)=\{X\in\mathcal{X}_n: f_n(X)< t-D_n\}.
	\label{mas2}$$
\end{definition}

The original sample $\mathcal{X}_n$ is divided into three subsamples, $\mathcal{X}_n^+(t)$, $\mathcal{X}_n^-(t)$ and $\mathcal{X}_n\setminus \left(\mathcal{X}_n^+(t)\cup\mathcal{X}_n^-(t)\right)$. From an intuitive point of view, $ \mathcal{X}_n^+(t)$ and $\mathcal{X}_n^-(t)$ should be contained in $G(t)$ and its complementary, respectively. This property is proved in Proposition \ref{mostramaisenGlambda}. In addition, Proposition \ref{alberto5} ensures that $\mathcal{X}_n^+(t)\neq\emptyset$. 	If $G(t)$ is nonconvex then it can be seen that, with probability one and for $n$ large enough, the set $\{\gamma>0: C_\gamma(\mathcal{X}_n^+(t))\cap \mathcal{X}_n^-(t)=\emptyset\}$ is nonempty and upper bounded. So, the estimator proposed in (\ref{estimadorr0hat}) is well-defined. In order to guarantee that the estimator satisfies these interesting and natural properties, two conditions on the kernel estimator $f_n$ of $f$ must be considered, see again \cite{r2} for more details:
\begin{description}
	\item[K. ]\begin{enumerate}
		\item The kernel function $K$ is a continuous kernel of order at least $p$ with bounded support and finite variation.
		\item The bandwidth $h$ is of the order $(\log{n}/n)^{1/(d+2p)}$.
	\end{enumerate}
\end{description}

\begin{proposition}\label{mostramaisenGlambda}Let $G(t)$ be a compact and nonempty level set. Under assumptions (A), (D) and (K), let $\mathcal{X}_n$ be a random sample generated by the density function $f$ and let $\mathcal{X}_n^+(t)$ and $\mathcal{X}_n^-(t)$ be as established in Definition \ref{jejejeje}. Then,
	\begin{equation*}\mathbb{P}\left(\mathcal{X}_n^+(t)\subset G(t),\mbox{ }\mathcal{X}_n^-(t)\subset G(t)^c,\mbox{ }\forall t \in[l,u],\mbox{ eventually}\right)=1.\end{equation*}
\end{proposition}

Proposition \ref{alberto5} uniformly bounds the distance between $G(t)$ and $\mathcal{X}_n^+(t)$  guaranteeing, in addition, that the set $\mathcal{X}_n^+(t)$ is nonempty with probability one and for $n$ large enough.

\begin{proposition}\label{alberto5}Let $G(t)$ be a compact, nonempty and nonconvex level set. Under assumptions (A), (D) and (K), let $\mathcal{X}_n$ be a random sample generated by the density function $f$ and let $\mathcal{X}_n^+(t)$ be as established in Definition \ref{jejejeje}. Then, for all $\epsilon>0$ it is verified that
	$$\mathbb{P}\left(\sup_{t\in[l,u]}\sup_{x\in G(t)}d(x,\mathcal{X}_n^+(t))\leq \epsilon, \mbox{ eventually}\right)=1.$$

\end{proposition}

Corollary \ref{xeneralizacionsoporte} shows, in particular, that $\mathcal{X}_n^+(t)$ is a consistent estimator for $G(t)$ in Hausdorff distance uniformly in $t$. In this point, it is important to remember a similar property for the support $S$ of the density $f$. The set of sample points $\mathcal{X}_n$ is a Hausdorff consistent estimator for $S$ too.

\begin{corollary}\label{xeneralizacionsoporte}Let $G(t)$ be a compact, nonempty and nonconvex level set. Under assumptions (A), (D) and (K), let $\mathcal{X}_n$ be a random sample generated by the density function $f$ and let $\mathcal{X}_n^+(t)$ be as established in Definition \ref{jejejeje}. Then, for all $\epsilon>0$ it is verified that
	$$\mathbb{P}\left(\sup_{t\in[l,u]} d_H( G(t),\mathcal{X}_n^+(t) )\leq \epsilon,  \mbox{ eventually}\right)=1.$$
\end{corollary}

\section{Consistency results}\label{3}

The consistency of the estimator for the smoothing parameter and the resulting estimator for density level sets are established in Sections \ref{smc} and \ref{edls}, respectively.
\subsection{Consistency of the estimator for the smoothing parameter}\label{smc}

In this section, the consistency for the estimator proposed in (\ref{estimadorr0hat}) will be established in Theorem \ref{co}. It is proved the uniform convergence in $t\in[l,u]$ for the estimator proposed for the parameter $r_0(t)$.

\begin{theorem}\label{co}Let $G(t)$ be a compact, nonempty and nonconvex level set. Under assumptions (A), (D) and (K), let $r_0(t)$ and $\hat{r}_0(t)$ be as established in Definitions \ref{r_0_t} and \ref{jejejeje}, respectively. Then, for any $\epsilon>0$,
	\begin{equation}\label{kjooo}
		\mathbb{P}\left(\sup_{t\in[l,u]}|\hat{r}_0(t)-r_0(t)|\leq \epsilon, \mbox{ eventually}\right)=1.
	\end{equation}
\end{theorem}

Theorem \ref{co} establishes the uniform consistency for $\hat{r}_0(t)$. The shape parameter $r_0(t)$ is a lower bound for the reach under mild geometric restrictions on $G(t)$. This problem has been recently considered for
a manifold, see \cite{aamari}, for references and examples.

\subsection{Consistency of the resulting estimator for density level sets}\label{edls}

Once the consistency for the estimator of the smoothing parameter $r_0(t)$ defined in
(\ref{estimadorr0levelsetu}) was studied, it would be natural to consider $C_{\hat{r}_0(t)}(\mathcal{X}_n^+(t))$
as an estimator for the level
set $G(t)$. However, the consistency can not be guaranteed in this case at least in general. We will propose $C_{r_n(t)}(\mathcal{X}_n^+(t))$ as the estimator of the level set $G(t)$ where $r_n(t)=\nu \hat{r}_0(t)$ for a fixed value $\nu\in (0,1)$. This reconstruction
of the theoretical level set presents interesting properties. Unlike plug-in estimators, the boundary of $C_{r_n(t)}(\mathcal{X}_n^+(t))$ is easy to handle since it is formed by arcs of balls of radius $r_n(t)$. Furthermore, $C_{r_n(t)}(\mathcal{X}_n^{+}(t))$ is uniformly consistent estimator in $t\in[l,u]$. The convergence rates are provided in Theorem \ref{principal}.

\begin{theorem}\label{principal}Let $G(t)$ be a compact, nonempty and nonconvex level set. Under assumptions (A), (D) and (K), let $\mathcal{X}_n$ be a random sample generated from a distribution with density function $f$, let $\mathcal{X}_n^+(t)$ be as established in Definition \ref{jejejeje} and let $r_n(t)=\nu \hat{r}_0(t)$ where $\nu\in (0,1)$ is a fixed number and $\hat{r}_0(t)$, defined in (\ref{estimadorr0hat}). Then, $$ \sup_{t\in [l,u]}d_H(C_{r_n(t)}(\mathcal{X}_n^+(t)),G(t))=O\left(\max\left\{ \left(  \frac{\log{n}}{n}\right)^{p/(d+2p)},\left(\frac{\log{n}}{n}\right)^{\frac{2}{d+1}}\right\}\right),$$almost surely.
	The same convergence order holds for $d_\mu(C_{r_n(t)}(\mathcal{X}_n^+(t)),G(t))$.
\end{theorem}

According to Theorem \ref{principal}, the new method proposed achieves minimax rates for Hausdorff metric and distance in measure, up to log factors, in most cases where the minimax rates are known, see \cite{mam}. Unlike granulometric smoothing method, when the smoothing parameter is unknown, the convergence rates do not incur a penalty term, see Theorem 3 in \cite{r2}. The rates obtained in Theorem \ref{principal} do no depend on any penalty term although $r_0(t)$ is a priori unknown and it is estimated in a data-driven way from $\mathcal{X}_n$.

Next, the estimation of the level set $L(\tau)$ defined in (\ref{conjuntonivel2}) will be considered. Theorem \ref{principal2} also establishes the previous convergence rates for the estimator $C_{r_n(\hat{f}_\tau)}(\mathcal{X}_n^+(\hat{f}_\tau))$ where $\hat{f}_\tau=\max\{t>0: \mathbb{P}_n(C_{r_n(t)}(\mathcal{X}_n^+(t)))\geq 1-\tau\}
$ and $\mathbb{P}_n$ denotes the empirical probability measure induced by $\mathcal{X}_n$.

\begin{theorem}\label{principal2}Let $L(\tau)$ be a compact, nonempty and nonconvex level set. Under assumptions (A), (D) and (K), let $\mathcal{X}_n$ be a random sample generated from a distribution with density function $f$, let $\mathcal{X}_n^+(t)$ be as introduced in Definition \ref{jejejeje}, let $r_n(t)=\nu \hat{r}_0(t)$ where $\nu\in (0,1)$ is a fixed number and $\hat{r}_0(t)$, defined in (\ref{estimadorr0hat}) and let $\hat{f}_\tau$ be the estimator for the threshold $f_\tau$ established in (\ref{umbral}). If $\underline{\tau}>\overline{\tau}$ are such that $l<f_{\underline{\tau}}$, $f_{\overline{\tau}}<u$ then
	$$ \sup_{\tau\in[\underline{\tau},\overline{\tau}]}(C_{r_n(\hat{f}_\tau)}(\mathcal{X}_n^+(\hat{f}_\tau)),L(\tau))=O\left(\max\left\{ \left(  \frac{\log{n}}{n}\right)^{p/(d+2p)},\left(\frac{\log{n}}{n}\right)^{\frac{2}{d+1}}\right\}\right),$$almost surely.
	The same convergence order holds for $d_\mu(C_{r_n(\hat{f}_\tau)}(\mathcal{X}_n^+(\hat{f}_\tau)),L(\tau))$.
\end{theorem}

\section{Data-driven estimation algorithm}\label{4}Most of the times, reconstruction of level sets $G(t)$ introduced in (\ref{conjuntonivel1}) has not interest for practical purposes. Since the practitioner usually unknowns the value of the level $t$, the most common alternative is then to establish the probability content $1-\tau$ of the level set to be estimated. Therefore, estimation of level sets $L(\tau)$ defined in (\ref{conjuntonivel2}) is considered in most of applications.

The first natural step to reconstruct $L(\tau)$ from $\mathcal{X}_n$ should be to determinate an estimator $\hat{f}_\tau$ of the threshold $f_\tau$. For instance, see \cite{hyn}. Then, it would be necessary to establish the sets $\mathcal{X}_n^+(\hat{f}_\tau)$ and $\mathcal{X}_n^-(\hat{f}_\tau)$ in order to calculate $\hat{r}_0(\hat{f}_\tau)$. However, Definition \ref{jejejeje} ensures that $\mathcal{X}_n^+(\hat{f}_\tau)$ and $\mathcal{X}_n^-(\hat{f}_\tau)$ rely on the sequence $D_n$. Therefore, calculating them is not straightforward. To solve this problem, a data-driven proposal for selecting $D_n$ is presented next.\vspace{1.75mm}\\
\underline{Data-driven selection of the sequence $D_n$:}\vspace{1.75mm}\\
Given $\tau$ and $\mathcal{X}_n$, let $p$ be a number verifying that $0<p<\min\{\tau,1-\tau\}$.\vspace{1mm}\\
As first step, it is necessary to calculate the ($\tau \pm p$)-quantiles of $f_n(\mathcal{X}_n)$. They are denoted by $\hat{f}_\tau^+$ and $\hat{f}_\tau^-$, respectively.\vspace{1mm}\\
From them, three subsets of the original sample $\mathcal{X}_n$ are defined:
$$\mathcal{X}_n^+(\hat{f}_\tau^+)=\{X\in\mathcal{X}_n:f_n(X)\geq \hat{f}_\tau^+\},\mbox{ }\mathcal{X}_n^-(\hat{f}_\tau^-)=\{X\in\mathcal{X}_n:f_n(X)< \hat{f}_\tau^-\} \mbox{ and}\vspace{1mm}$$
$$\mathcal{X}_n\setminus (\mathcal{X}_n^+(\hat{f}_\tau^+)\cup \mathcal{X}_n^-(\hat{f}_\tau^-)).$$	
Elements of the set $\mathcal{X}_n\setminus (\mathcal{X}_n^+(\hat{f}_\tau^+)\cup \mathcal{X}_n^-(\hat{f}_\tau^-))$ must be sorted by Euclidean distance to $\mathcal{X}_n^+(\hat{f}_\tau^+)$. The nearest points will be added to $\mathcal{X}_n^+(\hat{f}_\tau^+)$ until the number of points in $\mathcal{X}_n^+(\hat{f}_\tau^+)$ reachs the value $(1-\tau)n$. In this way, it is guaranteed that the proportion of sample points in $\mathcal{X}_n^+(\hat{f}_\tau)$ is at least $1-\tau$. The rest of points in $\mathcal{X}_n\setminus (\mathcal{X}_n^+(\hat{f}_\tau^+)\cup \mathcal{X}_n^-(\hat{f}_\tau^-))$ must be included in $\mathcal{X}_n^-(\hat{f}_\tau^-)$.\vspace{1mm}\\	
Finally, dichotomy algorithms can be used to compute
$$\hat{r}_0(\hat{f}_\tau)=\sup\{\gamma>0:C_\gamma(\mathcal{X}_n^+(\hat{f}_\tau^+))\cap \mathcal{X}_n^-(\hat{f}_\tau^-)=\emptyset\}$$from the already updated sets $\mathcal{X}_n^+(\hat{f}_\tau^+)$ and $\mathcal{X}_n^-(\hat{f}_\tau^-)$. See \cite{r00000} for a similar procedure. The practitioner
must select a maximum number of iterations $I$ and two initial points $r_m$ and $r_M$ with $r_m<r_M$ such that $C_{r_M}(\mathcal{X}_n^+(\hat{f}_\tau^+))\cap \mathcal{X}_n^-(\hat{f}_\tau^-)\neq\emptyset$ and $C_{r_m}(\mathcal{X}_n^+(\hat{f}_\tau^+))\cap \mathcal{X}_n^-(\hat{f}_\tau^-)=\emptyset$, respectively. Choosing a value close
enough to zero is sufficient to select $r_m$. On the other hand, if the convex hull of $\mathcal{X}_n^+(\hat{f}_\tau^+)$ does not meet $\mathcal{X}_n^-(\hat{f}_\tau^-)$ then we propose it as the estimator for the level set.

According to the previous steps, $C_{r_0(\hat{f}_\tau)}(\mathcal{X}_n^+(\hat{f}_\tau^+))$ contains a proportion of at least $1-\tau$ of points in $\mathcal{X}_n$. Of course, the value of $p$ must be chosen carefully taking values close to zero.

\subsection{Leukaemia data analysis}

Spatial clustering of rare diseases has
grown in recent years, in part prompted by increasing concerns over possible links between
disease and sources of environmental pollution.

The data set that will be studied in this work derives from the study in \cite{hen} and it is available in \cite{dig}. It contains 1221 pairs of points in Lancashire and Greater Manchester. Concretely, it
contains the residential coordinates for the 233 cases of diagnosed chronic granulocytic
leukemia registered between 1982 up to 1998 (inclusive), together with 988 controls.
For the selection of controls, population counts in each of the 8131 census enumeration
districts that make up the study-region, stratified by age and sex, were extracted from
the 1991 census. The counts were then used to obtain a stratified random sample of two
controls per case with coordinates given by their corresponding centroid coordinates
(slightly jittered to avoid coincident points).
\begin{figure}[h!]\vspace{-1.35cm}
	
	\hspace{2cm}\includegraphics[height=5.5cm,width=11cm]{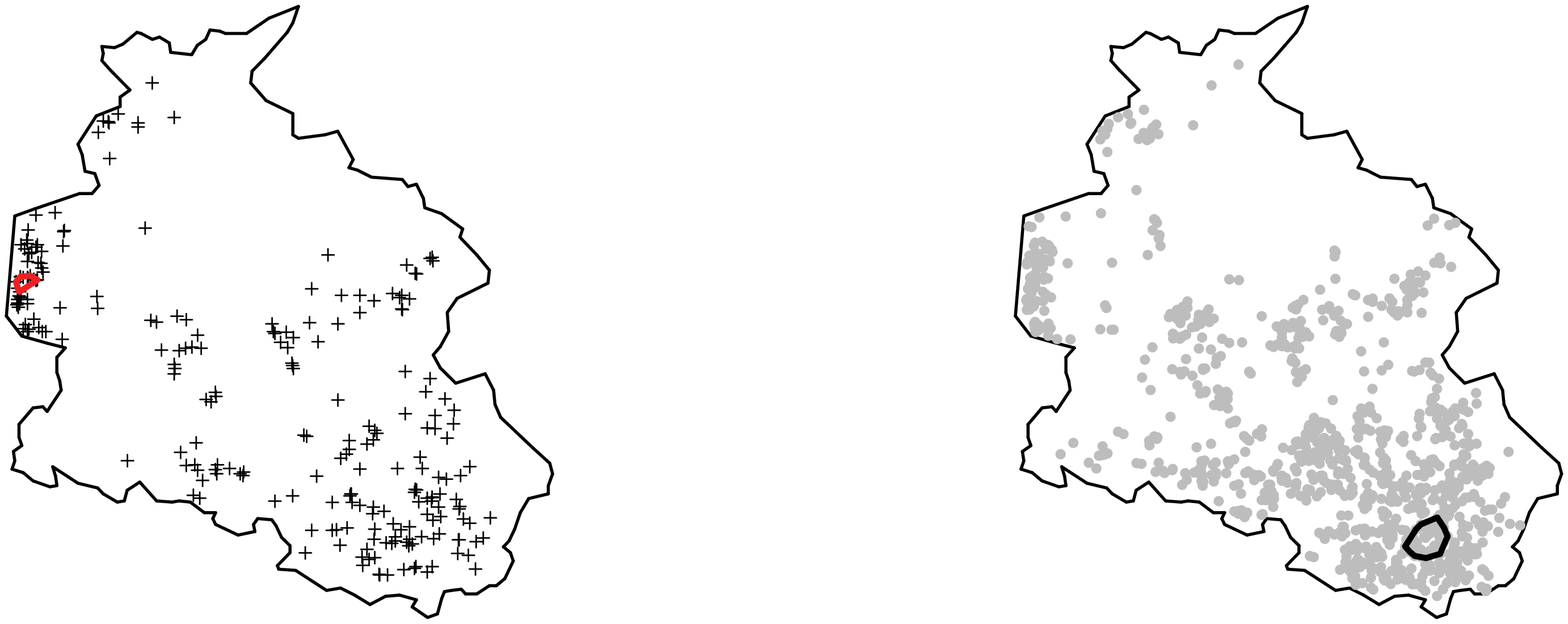}\vspace{-1.8cm}\\
	
	\hspace{2cm}\includegraphics[height=5.5cm,width=11cm]{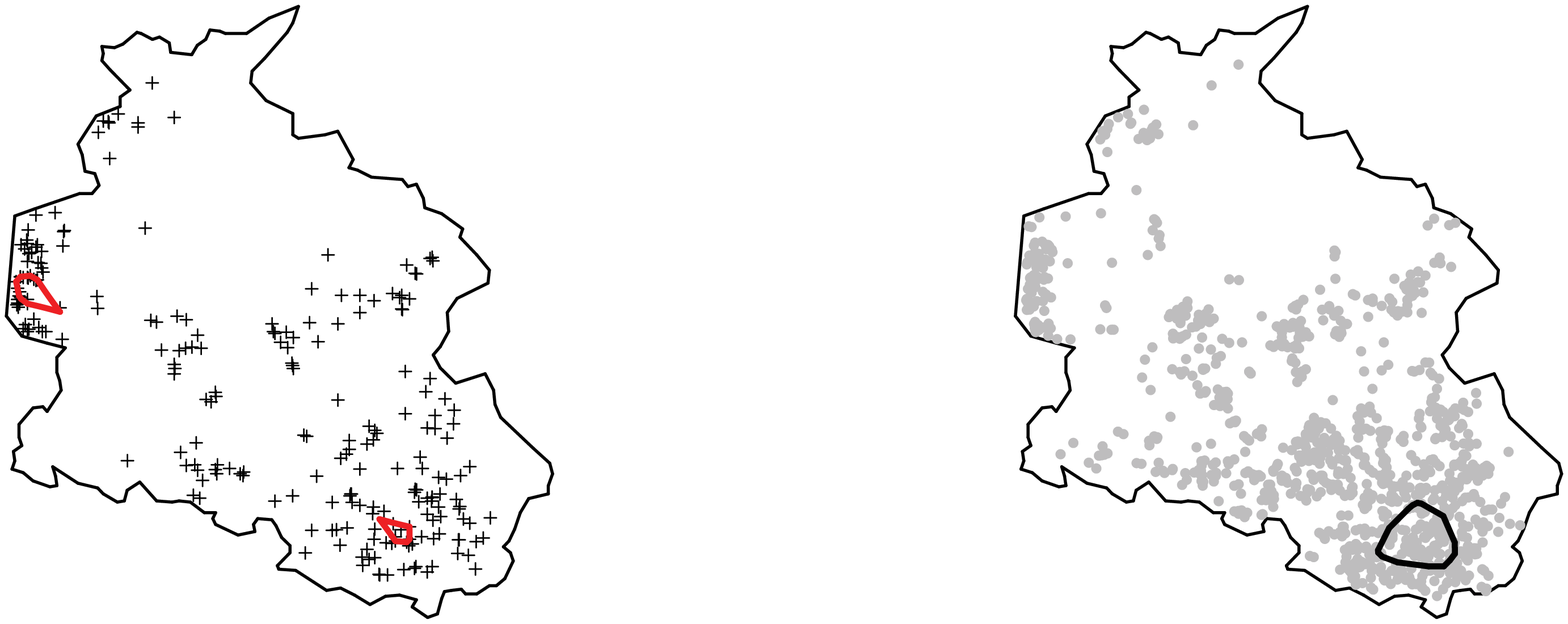}\vspace{-1.8cm}\\
	
	\hspace{2cm}\includegraphics[height=5.5cm,width=11cm]{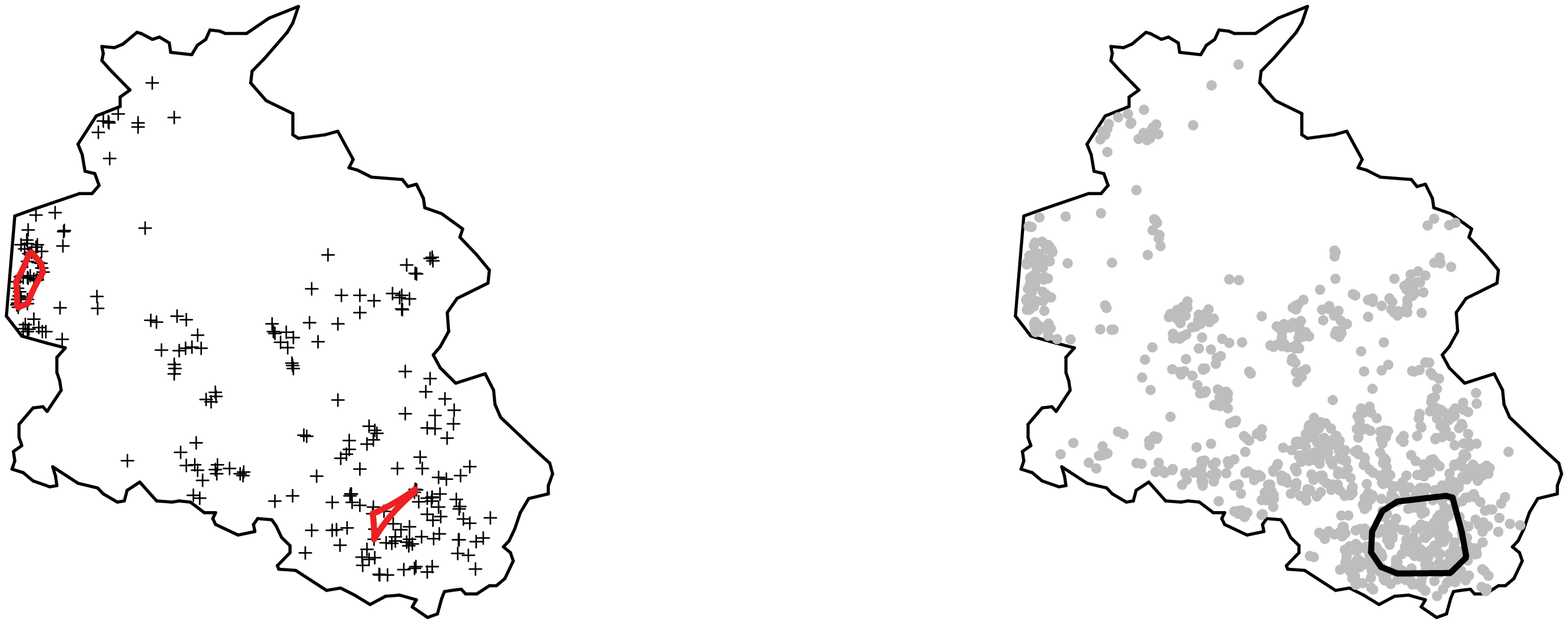}\vspace{-1.8cm}\\
	
	\hspace{2cm}\includegraphics[height=5.5cm,width=11cm]{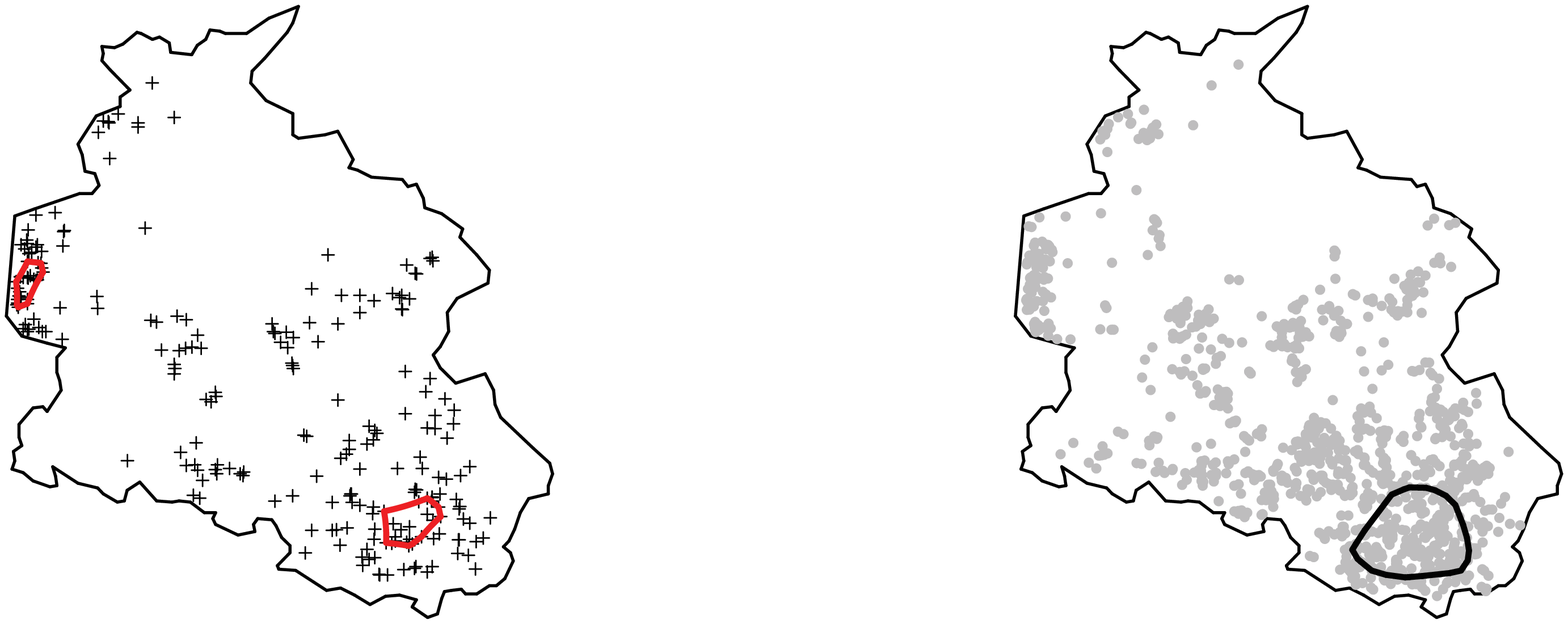}\vspace{-1.25cm}
	\caption{Level sets estimators for the distribution of 322 cases diagnosed of leukaemia (first column) and 988 controls (second column) on the North West of England with $\tau=0.95$ (first row), $\tau=0.9$ (second row), $\tau=0.85$ (third row) and $\tau=0.8$ (fourth row).}\vspace{-.5cm}
	\label{leucemiaest1}
\end{figure}

In Figure \ref{leucemiaest1}, the density level set estimator proposed in this work is shown for the samples of cases and controls fixing $p=0.01$ when $\tau=0.95$ and $\tau=0.9$ and, $p=0.1$ when $\tau=0.85$ and $\tau=0.8$. The multivariate generalization of the plug-in bandwidth selector in \cite{wand2} was used for estimating the corresponding density functions.

It is very interesting problem to examine whether the distribution of this kind of cancer mirrored that of the controls as a whole or whether there was evidence, as implied by concerned
local residents, of clustering. According to the estimations obtained, there exists an excess of case intensity over that of population. Greater Manchester is one of the largest metropolitan areas in the United Kingdom. However, Lancashire is a
non-metropolitan county that emerged during the Industrial Revolution as a major
commercial and industrial region. Therefore, there is evidence of clustering and the
leukaemia cases could be related to environmental and industrial factors.\vspace{-.1cm}

\section{Conclusions and extensions}\label{5}

Under $r-$convexity assumption, a fully data-driven estimator for the shape parameter $r$ is proposed. One advantage of the resulting level set estimator is that its geometric structure is easy to handle. In particular, its boundaries can be explicitly computed.

Furthermore, theoretical results are provided. Concretely, uniform consistency on the level for the estimator of $r$ is established. It is also proved that the obtained level set estimator achieves minimax convergence rates uniformly on the level for Hausdorff metric and Lebesgue measure, up to log factors. These rates do not rely on any penalty term as with other methods where the parameter $r$ is not estimated from data.

Computational aspects of the new estimator were also considered. In particular, a data-driven algorithm for selecting the sequence $D_n$ is provided. In addition, we finish this work showing a real application. We have checked that there can exists clustering evidence for leukaemia.

Finally, natural extensions of this work will be discussed. Although a solution for choosing the sequence $D_n$ was proposed, more sophisticated alternatives could be considered in future. Another important achievement would be to propose a nonparametric test for comparing two or more populations in general dimension. The test statistic could measure the discrepancy (for example, boundary distances) among the level set estimators of these populations. This test procedure could use explicitly the distance between boundaries of the estimated level sets. The simple geometric structure of estimators could be used to compute the procedure and calibrate the test using re-sampling schemes.

\section{Proofs}\label{6}$$\vspace{-1cm}$$

\textbf{Proof of Proposition \ref{mostramaisenGlambda}:}\vspace{-.23cm}\\

First, we will prove that,
$$\mathbb{P}(\mathcal{X}_n^+(t)\subset G(t),\mbox{ }\forall t \in[l,u],\mbox{ eventually})=1.$$For this, it is enough to show
\begin{equation}\label{jo}
	\mathbb{P}\left(\sup_{z\in G(t)^c} f_n(z)<t+M\left(\frac{\log{n}}{n}\right)^{p/(d+2p)},\mbox{ }\forall t \in[l,u],\mbox{ eventually}\right)=1.
\end{equation}Then, let $z\in G(t)^c$ and $t\in[l,u]$. Two cases are considered: $z\in C$ or $z\in C^c$, see Proposition \textcolor{blue}{A.1} in Appendix A for details about the compact set $C$.
\begin{enumerate}
	\item Let $z\in C^c$. Since $z\in G(t)^c$ then $z\notin G(l)$ because $G(l)\setminus \interior(G(u))\subset C$. Therefore, according to Proposition \textcolor{blue}{A.2} in Appendix A, with probability one and for $n$ large enough,
	$$\sup_{z\in G(t)^c\cap C^c}f_n(z)\leq \sup_{y\in G(l)^c\cap C^c}f_n(y)<l-\frac{w}{2}<l,$$where $w$ denotes a positive constant. Therefore,
	$$\mathbb{P}\left(\sup_{z\in G(t)^c\cap C^c}f_n(z)<l,\mbox{ }\forall t \in[l,u],\mbox{ eventually}\right)=1,$$and since $l<t+D_n$ for all $t\in [l,u]$,
	$$\mathbb{P}\left(\sup_{z\in G(t)^c\cap C^c}f_n(z)<t+D_n,\mbox{ }\forall t \in[l,u],\mbox{ eventually}\right)=1.$$
	\item According to Proposition \textcolor{blue}{A.1} in Appendix A, there exists $N>0$ such that
	\begin{equation}\label{hhhhhhhhhhhhhhhhhhhhhhhh}
		\sup_{C}|f_n-f|\leq N  \left(  \frac{\log{n}}{n}\right)^{p/(d+2p)},\mbox{ almost surely}.
	\end{equation}Let $z\in C$. Since $z\notin G(t)$ then $f(z)<t$. Taking into account (\ref{hhhhhhhhhhhhhhhhhhhhhhhh}), with probability one and for large $n$, it is verified that
	$$f_n(z)\leq|f_n(z)-f(z)|+|f(z)|<\sup_C|f_n-f|+t.$$Hence, for all $t\in[l,u]$,
	$$\sup_{z\in G(t)^c\cap C}f_n(z)\leq \sup_C |f_n-f|+t.$$If $M\geq N$,
	$$f_n(z) < t+M\left(  \frac{\log{n}}{n}\right)^{p/(d+2p)}=t +D_n,\mbox{ almost surely}.$$This concludes the proof of (\ref{jo}).
\end{enumerate}In a similar way, it can be proved that
\begin{equation*}
	\mathbb{P}(\mathcal{X}_n^-(t)\subset G(t)^c,\mbox{ }\forall t \in[l,u],\mbox{ eventually})=1.\qedhere
\end{equation*}

%


$\vspace{-.24cm} $

\textbf{Proof of Proposition \ref{alberto5}:}\vspace{-.23cm}\\

Let $\epsilon>0$. It is clear that it is enough to show the result for $\epsilon$ small enough. Nexts steps complete the proof:

\begin{enumerate}
	\item[Step 1.] Let $x\in G(t)$. Under (A), a ball of radius $m/k$ rolls freely in $G(t)$ and $\overline{G(t)^c}$ for all value of the threshold $t\in[l,u]$. According to Lemma 1 in \cite{r1}, if $\epsilon\leq m/k$,
	$$\exists B_{\frac{\epsilon}{2}}(y)\subset B_{\epsilon}(x)  \mbox{ such that } B_{\frac{\epsilon}{2}}(y)\subset G(t).$$Define $B_t^x=B_{\epsilon/4}(y)$. Obviously, $B_t^x\subset G(t) $. In addition, it verifies that
	$$B_t^x\subset G(t)\ominus \frac{\epsilon}{4}B_1[0]\subset G(l)\ominus (\epsilon/4)B_1[0],\mbox{ }\forall t\in [l,u]\mbox{ and }\forall x\in G(t)$$since, for all $z\in B_t^x$, $z+(\epsilon/4)B_1[0]\subset B_{\epsilon/2}(y)\subset G(t)$. On the other hand and considering Proposition \textcolor{blue}{A.3} (b) in Appendix A, for $\epsilon$ small enough and $T=\epsilon m/8$,
	$$G(t)\ominus \frac{\epsilon}{4}B_1[0] \subset G(t+T).$$Therefore,
	$$B_t^x\subset G(t)\ominus \frac{\epsilon}{4}B_1[0]\subset G(t+T).\vspace{-.2cm}$$			
	\item[Step 2.] Let $\mathcal{F}=\{B_t^x:t\in[l,u],\mbox{ }x\in G(t)\}$. According to the previous comments, if $T>0$ is small enough then $$B_t^x\subset G(t+T)\subset G(l), \mbox{ }\forall t\in[l,u],\mbox{ }\forall x\in G(t).$$On the other hand, under (A), the level set $G(l)$ is bounded since $G(u+\zeta)$ is bounded and $\overline{G(l-\zeta)}\setminus \interior (G(u+\zeta))\subset U$ where $U$ is a bounded set too. As consequence, $\overline{G(l-\zeta)}$ is bounded and, therefore, $G(l)\subset\overline{G(l-\zeta)}$ too. Then, $G(l)\ominus (\epsilon/4)B_1[0]$ is also bounded and compact. Therefore, there exists a finite cover for $G(l)\ominus (\epsilon/4)B_1[0]$ of balls of radius, for instance, $\epsilon/10$. Therefore, there exist $z_1,\cdots,z_s\in G(l)\ominus (\epsilon/4)B_1[0]$ such that
	$$G(l)\ominus (\epsilon/4)B_1[0]\subset \bigcup_{i=1}^s B_{\frac{\epsilon}{10}}(z_i).$$Then, for all $B_t^x=B_{\epsilon/4}(y)\in \mathcal{F}$ where $y\in  G(l)\ominus (\epsilon/4)B_1[0]$,
	$$\exists z_j\in \{z_1,\cdots,z_s\}\mbox{ such that }\|z_j-y\|<\frac{\epsilon}{10}.$$Next, we will prove that the ball $B_{\epsilon/10}(z_j)\subset B_t^x$. Let $z\in B_{\epsilon/10}(z_j)$, it is satisfied that
	$$\|z-y\|\leq\| z-z_j\|+\|z_j-y\| < \frac{\epsilon}{10}+ \frac{\epsilon}{10}=\frac{\epsilon}{5}<\frac{\epsilon}{4}.$$
	As consequence, if a ball in $\mathcal{F}$ does not meet $\mathcal{X}_n$ then there exist a ball $B_{\epsilon/10}(z_i)$ such that $B_{\epsilon/10}(z_i)\cap \mathcal{X}_n=\emptyset$. Therefore, we can write
	\begin{equation}\label{hu245}\mathbb{P}(\exists t\in[l,u],\mbox{ }x\in G(t):\mathcal{X}_n\cap B_t^x=\emptyset)\leq \sum_{i=1}^s \mathbb{P}\left(\mathcal{X}_n\cap B_{\frac{\epsilon}{10}}(z_i)=\emptyset\right).\end{equation}
	In addition, since $z_i\in G(l)\ominus (\epsilon/4)B_1[0]$ for all $i \in \{1,...,s\}$ it is satisfied that $B_{\epsilon/10}(z_i)\subset G(l)$ for all $i\in\{1,...,s\}$. Then, we can assume that
	\begin{equation}\label{hu}
		f(z)>l-\zeta\mbox{ for all }z\in  B_{\frac{\epsilon}{10}}(z_i),\mbox{ }i=1,\cdots,s.
	\end{equation}
	\begin{figure}[h!]\vspace{-2.1cm}
		\begin{pspicture}(-1.6,0)(14,4.5)
		\scalebox{.6}{\psccurve[showpoints=false,fillstyle=solid,fillcolor=white,linecolor=black,linewidth=0.2mm,linearc=3](5,1)(5,4)(8,3.5)(10,4.5)(10,1)
			\rput(0.87,.3){\scalebox{0.9}[0.85]{\psccurve[showpoints=false,fillstyle=solid,fillcolor=white,linecolor=gray,linewidth=0.2mm,linearc=3](5,1)(5,4)(8,3.5)(10,4.5)(10,1)}}
			\pscircle[ linecolor=black,linewidth=0.2mm,linearc=3,linestyle=dashed,dash=3pt 2pt](9.7,4.3){.9}
			\psdots*[dotsize=2.5pt](9.7,4.3)
			\pscircle[linearc=0.25,linecolor=gray,linewidth=0.2mm, linestyle=dashed,dash=3pt 2pt,fillstyle=crosshatch*,fillcolor=gray,hatchcolor=white,hatchwidth=1.pt,hatchsep=.4pt,hatchangle=0](9.7,3.8){.25}
			\psdots[dotsize=2.pt,linecolor=gray](9.7,3.8)}
		\rput(7,2.7){$\tiny{B_\epsilon(x)}$}
		\rput(2.2,1.05){$G(t)$}
		\rput(3.8,1.5){\textcolor[rgb]{0.52,0.52,0.52}{$G(t+T)$}}
		\end{pspicture}\vspace{-.3cm}
		\caption{Elements in proof of Proposition \ref{alberto5}. $G(t)$ in black, $G(t+T)$ in gray, $B_\epsilon(x)$ in black and $B_t^x$ in gray.}\vspace{-.3cm}\label{k}
	\end{figure}
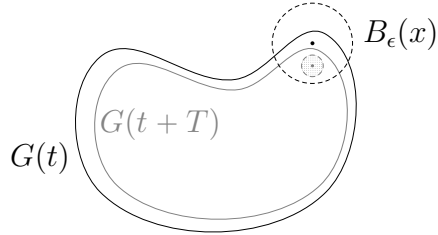
	\item[Step 3.] Given this, Borel-Cantelli's Lemma shows
	$$\mathbb{P}(\exists t\in[l,u],\mbox{ }x\in G(t):\mathcal{X}_n\cap B_t^x=\emptyset,\mbox{ infinitely often})=0.$$Using the same reasoning as in the Step 2, it is enough to analyze if $$\mathbb{P}\left(\mathcal{X}_n\cap B_{\frac{\epsilon}{10}}(z_i)=\emptyset,\mbox{ infinitely often}\right)=0.$$Therefore, we must to prove that $$\sum_{n=1}^\infty \mathbb{P}\left(\mathcal{X}_n\cap B_{\frac{\epsilon}{10}}(z_i)=\emptyset\right)<\infty.$$It is easy to check that
	$$\mathbb{P}\left(\mathcal{X}_n\cap B_{\frac{\epsilon}{10}}(z_i)=\emptyset\right)\leq  e^{-n\mathbb{P}\left(X_1\in B_{\frac{\epsilon}{10}}(z_i)\right)}$$and, using (\ref{hu}), we can ensure that $\mathbb{P}\left(X_1\in B_{\frac{\epsilon}{10}}(z_i)\right)>0$. Then,
	$$\sum_{n=1}^\infty \mathbb{P}\left(\mathcal{X}_n\cap B_{\frac{\epsilon}{10}}(z_i)=\emptyset\right)\leq \sum_{n=1}^\infty e^{-n \mathbb{P}\left(X_1\in B_{\frac{\epsilon}{10}}(z_i)\right)}<\infty.$$
	\item[Step 4.] According to Step 3, with probability one, there exists $n_0$ such that for all $t\in[l,u]$ and for all $x\in G(t)$,
	$$ \mathcal{X}_n\cap B_{t}^x\neq \emptyset,\mbox{ }\forall n\geq n_0.$$Then, there exists $n_0$ such that for all $t\in[l,u]$ and for all $x\in G(t)$,
	$$ \exists x_i\in \mathcal{X}_n\cap B_{t}^x \subset \mathcal{X}_n\cap B_\epsilon(x),\mbox{ }\forall n\geq n_0.$$Therefore, it only remains to prove that $x_i\in \mathcal{X}_n^+(t)$.
	
	Two cases are considered: $x_i\in C$ and $x_i\notin C$.
	\begin{enumerate}
		\item Let $x_i\in C$.  According to Proposition \textcolor{blue}{A.1} in Appendix A, there exists $N>0$ such that, with probability one,
		$$
		\sup_{C}|f_n-f|\leq N  \left(  \frac{\log{n}}{n}\right)^{p/(d+2p)} .
		$$If $D_n= M\left(  \frac{\log{n}}{n}\right)^{p/(d+2p)}$ with $M\geq N$ then $\lim_{n\rightarrow \infty}D_n=0$. So, fixed $T/2>0$ (see step 1 in this proof),
		$$\exists n_1\in\mathbb{N}\mbox{ such that } D_n<T/2, \forall n\geq n_1.$$Then,
		$$|f_n(x_i)-f(x_i)|\leq \sup_{C}|f_n-f|\leq D_n<T/2,\mbox{ }\forall n\geq\max\{n_0,n_1\}.$$Therefore, since $x_i\in B_t^x\subset G(t+T)$,
		$$f_n(x_i)\geq f(x_i)-D_n\geq t+T-D_n>t+T-\frac{T}{2}=t+\frac{T}{2}\geq t+D_n.$$

		\item If $x_i\notin C$ then, since $f(x_i)\geq t+T>t\geq l$ and $G(l)\setminus \interior(G(u))\subset C $, we have $x_i\in \interior(G(u))\subset G(u)$. Then,
		$x_i\in G(u)\cap C^c$. According to Proposition \textcolor{blue}{A.2} in Appendix A for a certain $w>0$, with probability one,
		$$\exists n_2 \mbox{ such that  }f_n(z)\geq u +\frac{w}{2},\mbox{ }\forall z\in G(u)\cap C^c\mbox{ and }\forall n\geq n_2.$$For $D_n$ fixed previously, $\lim_{n\rightarrow \infty}D_n=0$. So, given $w/2>0$,$$\exists n_3\in\mathbb{N}\mbox{ such that } D_n<w/2, \forall n\geq n_3.$$Therefore, since $t\leq u$,
		\[f_n(x_i)\geq u+\frac{w}{2}\geq t + D_n,  \mbox{ }\forall n\geq\max\{n_0,n_2,n_3\}.\qedhere\]
	\end{enumerate}
\end{enumerate}

$\vspace{-.24cm} $

\textbf{Proof of Corollary \ref{xeneralizacionsoporte}:}\vspace{-.23cm}\\

The proof is a straightforward consequence of Propositions \ref{mostramaisenGlambda} and \ref{alberto5}.\qedhere

$\vspace{-.24cm} $

\textbf{Proof of Theorem \ref{co}:}\vspace{-.23cm}\\

Previously, it is necessary to prove some preliminary results. First, we will prove that the estimator $\hat{r}_0(t)$ is greater than the real value $r_0(t)$ for a fixed value of $t$, with probability one and $n$ large enough. The behaviour of the right-sided limit of $r_0(t)$ is analyzed in Proposition \ref{continuidadlambad}.

\begin{proposition}\label{roooo}Let $G(t)$ be a compact, nonempty and nonconvex level set. Under assumptions (A), (D) and (K), let $r_0(t)$ and $\hat{r}_0(t)$ be as established in Definitions \ref{r_0_t} and \ref{jejejeje}, respectively. Then,
	$$\mathbb{P}(\hat{r}_0(t)\geq r_0(t),\mbox{ }\forall t\in [l,u],\mbox{ eventually})=1.$$		
\end{proposition}

\begin{proof}Proposition \ref{alberto5} guarantees that, with probability one,
	$$\exists n_1\in\mathbb{N}\mbox{ such that } \mathcal{X}_n^+(t)\neq\emptyset, \mbox{ }\forall t\in[l,u], \mbox{ }\forall n\geq n_1.$$In addition, Proposition \ref{mostramaisenGlambda} guatantees that, with probability one,
	$$\exists n_2\in\mathbb{N}\mbox{ such that }\mathcal{X}_n^+(t)\subset G(t) \mbox{ and }\mathcal{X}_n^-(t)\subset G(t)^c, \mbox{ }\forall t\in[l,u],\mbox{ }\forall n\geq n_2.$$Since $G(t)$ is $r_0(t)-$convex it is verified that $\mathcal{X}_n^+(t)\subset  C_{r_0(t)}(G(t))=G(t)$ and $\mathcal{X}_n^-(t)\subset G(t)^c$, with probability one, for $n$ large enough. Hence, $\mathcal{X}_n^+(t)\cap \mathcal{X}_n^-(t)=\emptyset$ and, with probability one,
	\[r_0(t)\leq \sup\{\gamma>0: C_\gamma(\mathcal{X}_n^+(t))\cap \mathcal{X}_n^-(t)=\emptyset\}=\hat{r}_0(t),$$
	$$\mbox{ }\forall t\in [l,u],\mbox{ }\forall n\geq n_0=\max\{ n_1,n_2\}.\qedhere\]
\end{proof}

\begin{proposition}\label{continuidadlambad}Let $G(t)$ be a compact, nonempty and nonconvex level set and let $r_0(t)$ be as established in Definition \ref{r_0_t}. Under (A), given $\epsilon>0$, there exists $\delta>0$ such that for all $t,\mbox{ }\overline{t}\in [l,u]$ verifying $0<\overline{t}-t<\delta$ it is satisfied that $r_0(\overline{t})\leq r_0(t)+\epsilon$.
\end{proposition}
\begin{proof}Given $\epsilon>0$, let $t\in[l,u]$ and let $r^*>r_0(t)$. Then, it is defined $r=r^*+\epsilon$.
	
	According to Proposition \textcolor{blue}{B.1} in Appendix B for $r^*$, there exists an open ball $A_t=B_{\rho_t}(c_t)$ such that $A_t\cap G(t)=\emptyset$ and $A_t\subset C_{r^*}(G(t))$.

	Let $\delta=\min\{ \epsilon m/2, mc/2 \}$ be a number that does not rely on $t$. If $0<\overline{t}-t\leq \delta$, Proposition \textcolor{blue}{B.4} in Appendix B guarantees that $B_{\rho_t/2}(c_t)\subset C_{r^*+\epsilon}(G(\overline{t}))=C_{r}(G(\overline{t}))$.
	
	Since $A_t\cap G(t)=\emptyset$ and $\overline{t}>t$, it is satisfied that $A_t\cap G(\overline{t})=\emptyset$. Then, $G(\overline{t})\subsetneq C_r(G(\overline{t}))$. Therefore, $r_0(\overline{t})\leq r= r^*+\epsilon$ for all $r^*>r_0(t)$. As consequence, $r_0(\overline{t})\leq r_0(t)+\epsilon$.
\end{proof}

Next, let $\epsilon>0$, $t\in[l,u]$ and $r=r_0(t)+\epsilon$. Two cases are distinguished for proving the uniform convergence in $[l,u]$ of the estimator for the shape index. In Case 1, we will prove that, with probability one and for $n$ large enough, there exists $\delta^*>0$ such that if $0<t-\overline{t}\leq \delta^*$ then $\hat{r}_0(\overline{t})\leq  r_0(\overline{t})+2\epsilon.$  In Case 2, we will prove that there exists $\overline{\delta}>0$ such that if $0<\overline{t}-t\leq \overline{\delta}$ then $\hat{r}_0(\overline{t})\leq  r_0(\overline{t})+4\epsilon.$\vspace{-.5cm}

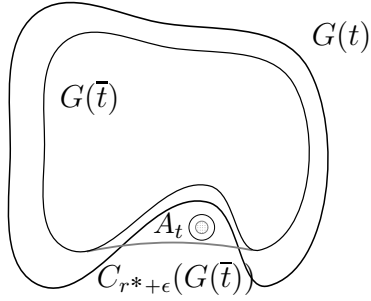
\begin{figure*}[h!]\vspace{-3.85cm}
	\begin{pspicture}(-3.99,-1.9)(10,7.5)
	\scalebox{0.55}{
		\rput(-1.1,-.85){	\scalebox{1.2}[1.3]{	 \psccurve[showpoints=false,fillstyle=none,fillcolor=white,linecolor=black,linewidth=0.3mm,linearc=3](1.5,0)(1,2.5)(1.5,5)(4,4.85)(7,4)(6.2,-0.05)(5,1.5)}}
		\psarc[showpoints=false,fillstyle=none,fillcolor=white,linecolor=gray,linewidth=0.45mm,linearc=3](4.1,-9.867) {10}{78.9}{102}
		
		\psccurve[showpoints=false,fillstyle=none,fillcolor=white,linecolor=black,linewidth=0.3mm,linearc=3](1.5,0)(1,2.5)(1.5,5)(4,4.85)(7,4)(6.2,-0.05)(5,1.5)
		
		\pscircle[linearc=0.25,linecolor=black,linewidth=0.2mm,linestyle=solid,dash=3pt 2pt](4.8,.5){.32}
		\pscircle[linearc=0.25,linecolor=gray,linewidth=0.2mm,linestyle=solid,dash=3pt 2pt,fillstyle=crosshatch*,fillcolor=gray,hatchcolor=white,hatchwidth=1.2pt,hatchsep=.5pt,hatchangle=0](4.8,.5){.16}}
	\rput(1.15,2){$G(\overline{t})$}
	\rput(2.2,0.3){\small{$A_t$}}
	\rput(2.3,-.4){$C_{r^*+\epsilon}(G(\overline{t}))$}
	\rput(4.5,2.8){$G(t)$}
	\end{pspicture}\vspace{-1.39cm}\caption{Elements of proof in Proposition \ref{continuidadlambad}. $B_{\rho_t/2}(c_t)$ in gray color. }\vspace{-.75cm}\label{ooskkkssokkkdkd}
\end{figure*}
$ $\\\underline{Case 1:} Let $\overline{t}<t$ with $\overline{t} \in [l,u]$.
First, we will prove that $\hat{r_0}(\overline{t})\leq r_0(t)+\epsilon$ if $t-\overline{t}\leq \delta_1$ with $\delta_1$ depending on $t$. Secondly, we will prove that $ r_0(t)\leq  r_0(\overline{t})+\epsilon$ if $t-\overline{t}\leq \delta_2$.\\



Lemma \textcolor{blue}{B.2} in Appendix B, ensures that there exists an open ball $B_t$ of radius $\gamma_t$ verifying, with probability one and for $n$ large enough,
$$B_t\subset  C_{r}(\mathcal{X}_n^+(\overline{t})).$$If $t-\overline{t}\leq\delta_1=\min\{\frac{mc}{2}, m \gamma_t\}$, Lemma \textcolor{blue}{B.3} in Appendix B ensures, with probability one and for $n$ large enough,
$$\emptyset\neq\mathcal{X}_n^-(\overline{t})\cap B_t \subset C_{r}(\mathcal{X}_n^+(\overline{t})).$$Therefore,
$$\emptyset\neq\mathcal{X}_n^-(\overline{t})\cap  C_{r}(\mathcal{X}_n^+(\overline{t}))$$and, with probability one and for $n$ large enough,
$$\hat{r_0}(\overline{t})\leq r=r_0(t)+\epsilon.$$
According to Proposition \ref{continuidadlambad}, given $\epsilon>0$,
$$\exists \delta_2>0 \mbox{ such that }\forall t,\mbox{ }\overline{t}\mbox{ verifying} 0<t-\overline{t}<\delta_2\mbox{ then }r_0(t)\leq  r_0(\overline{t})+\epsilon.$$Remember that $\delta_2$ does not rely on $\overline{t}$.\\$ $\\
If $t-\overline{t}\leq \delta^*=\min\{\delta_1,\delta_2\}$ then $\hat{r}_0(\overline{t})\leq  r_0(\overline{t})+2\epsilon.$\\
$ $\\
\underline{Case 2:} Let $\overline{t}>t$ with $\overline{t} \in [l,u]$. According to Proposition \textcolor{blue}{A.3} (b) in Appendix A, if $ \overline{t}-t\leq mc/2$ then
\begin{equation}\label{lewl2b}
	G(\overline{t})\subset G(t)\subset G(\overline{t})\oplus\frac{2}{m}(\overline{t}-t)B_1[0].
\end{equation}As consequence,
\begin{equation}\label{hu2}d_H(G(t),G(\overline{t}))\leq \frac{2}{m}(\overline{t}-t).\end{equation}
According to Proposition \textcolor{blue}{B.1} in Appendix B, given $r=r_0(t)+\epsilon$, there exists an open ball $A_t=B_{\rho_t}(c_t)$ with $A_t\cap G(t)=\emptyset$ and $A_t\subset C_r(G(t))$. Taking $\overline{t}-t\leq \min\{\frac{\epsilon m}{6},mc/2\}= \overline{\delta}_1$, Proposition \textcolor{blue}{B.5} in Appendix B ensures that $B_{\rho_t/4}(c_t)$ verifies that, with probability one,
$$\exists n_0\in\mathbb{N}\mbox{ such that }B_{\rho_t/4}(c_t)\subset C_{r+2\epsilon}(\mathcal{X}_n^+(\overline{t})),\mbox{ }\forall n\geq n_0.$$
According to Lemma \textcolor{blue}{B.6} in Appendix B, with probability one and for $n$ large, it is verified that $B_{\rho_t/4}(c_t)\cap \mathcal{X}_n^-(\overline{t})\neq\emptyset$. So, $\hat{r}_0(\overline{t})\leq r+2\epsilon=r_0(t)+3\epsilon$.\\		
Given $\epsilon>0$ established at the beginning, Proposition \ref{continuidadlambad} guarantees that
$$\exists \overline{\delta}_2>0 \mbox{ such that } t-\overline{t}<\overline{\delta_2} \mbox{ then } r_0(t)\leq r_0(\overline{t})+\epsilon.$$Therefore, if $\overline{t}-t\leq \overline{\delta}=\min\{\overline{\delta}_1,\overline{\delta}_2\}$ then, with probability one and for $n$ large enough, $\hat{r}_0(\overline{t})\leq  r_0(\overline{t})+4\epsilon.$\\

Next, we will take into account Cases 1 and 2. For each $t\in[l,u]$, we consider $\delta_t=\min\{\delta^*,\overline{\delta},\zeta/2\}$. See Assumption A for details on $\zeta$. Then,
$$\{(t-\delta_t,t+\delta_t): t\in[l,u]\}$$is an open covering of the compact $[l,u]$.
Therefore, there exists a finite subrecovering. Then, there exists $t_1,...,t_k\in[l,u]$ and the corresponding $\delta_{t_1}=\delta_1,...,\delta_{t_k}=\delta_k$ verifying
$$[l,u]\subset\bigcup_{j=1}^k (t_j-\delta_j,t_j+\delta_j).$$
Let $\overline{t}\in [l,u]$. For each $j=1,...,k$, with probability one, there exists $n_j\in\mathbb{N}$ such that if $\overline{t}\in (t_j-\delta_j,t_j+\delta_j)$ then
$$\hat{r}_0(\overline{t})\leq   r_0(\overline{t})+4\epsilon,\mbox{ }\forall n\geq n_j.$$Therefore, since
$$\forall\overline{t}\in[l,u]\mbox{ }\exists j\in\{1,...,k\}\mbox{ such that }\overline{t}\in (t_j-\delta_j,t_j+\delta_j)$$it is verified
$$\hat{r}_0(\overline{t}) \leq r_0(\overline{t})+4\epsilon,\mbox{ }\forall n\geq \max\{n_1,n_2,...,n_k\}=N.$$In addition, since $\delta_j\leq\zeta/2$ and $t_j\in [l,u]$ for all $j=1,...,k$, it is verified that$$\max\{t_j+\delta_j, j=1,...,k\}\leq u+\zeta/2\mbox{ and }\min\{t_j-\delta_j,j=1,...,k\}\geq l-\zeta/2.$$
Taking Proposition \ref{alberto5} into account, it can be proved easily that, with probability one and for $n$ large enough, $\mathcal{X}_n^+(\overline{t})\neq\emptyset$ and $\mathcal{X}_n^-(\overline{t})\neq\emptyset$.
Proposition \ref{roooo} guarantees that, with probability one and for $n$ large enough, $r_0(\overline{t})\leq \hat{r}_0(\overline{t})$. Therefore, with probability one, for $n$ large enough,
\begin{equation}\label{jit}
	r_0(\overline{t})\leq \hat{r}_0(\overline{t})\leq  r_0(\overline{t})+4\epsilon.
\end{equation}
Then, with probability one and for $n$ large enough,
$$\forall\overline{t}\in[l,u]\mbox{ it is satisfied that }|r_0(\overline{t})-\hat{r}_0(\overline{t})|<4\epsilon.$$Therefore, with probability one,
\[\lim_{n\rightarrow\infty}\sup_{\overline{t}\in[l,u]}|r_0(\overline{t})-\hat{r}_0(\overline{t})|=0.\qedhere\]

$\vspace{-.24cm} $

\textbf{Proof of Theorem \ref{principal}:}\vspace{-.23cm}\\

Next, some auxiliary proofs are presented. Proposition \ref{llllll} establishes that the estimator $C_{r_n(t)}(\mathcal{X}_n^+(t))$ is contained in the theoretical level set with probability one and for $n$ large enough.

\begin{proposition}\label{llllll}Let $G(t)$ be a compact, nonempty and nonconvex level set. Under assumptions (A), (D) and (K), let $\mathcal{X}_n$ be a random sample generated from a distribution with density function $f$, let $r_0(t)$ and $\hat{r}_0(t)$ be as established in Definitions \ref{r_0_t} and \ref{jejejeje}, respectively. Let $\nu\in(0,1)$ be a fixed number and $r_n(t)=\nu\hat{r}_0(t)$. Then,
	$$\mathbb{P}( C_{r_n(t)}(\mathcal{X}_n^+(t))\subset G(t),\mbox{ }\forall t\in[l,u],\mbox{ eventually})=1.$$
\end{proposition}

\begin{proof}According to Proposition \ref{mostramaisenGlambda},
	$$\mathbb{P}(\mathcal{X}_n^+(t)\subset G(t),\mbox{ }\forall t\in[l,u],\mbox{ eventually})=1.$$According to Theorem \ref{co}, with probability one,
	$$\lim_{n\rightarrow \infty}\sup_{t\in[l,u]}| r_n(t)-\nu r_0(t)|=0.$$Therefore,
	$$\mathbb{P}(r_n(t)\leq r_0(t),\mbox{ }\forall t\in[l,u],\mbox{ eventually})=1$$and, as consequence,
	$$\mathbb{P}( C_{r_n(t)}(\mathcal{X}_n^+(t))\subset   G(t),\mbox{ }\forall t\in[l,u],\mbox{ eventually})=1.\qedhere$$\end{proof}

At this point, it is necessary to introduce some auxiliary sets in order to obtain the convergence rates of the resulting estimator for the level set, see Definitions \ref{auxiset} and \ref{auximuestra}. Really, these new sets are subsets of the original level set $G(t)$ and the sample $\mathcal{X}_n$, respectively. Notice that both are defined from the theoretical density function $f$ so they are unknown. The kernel estimator $f_n$ is not considered. On the other hand and although they depend on some parameters like $n$ in their definition, this fact is not reflected in their names for simplicity in the exposition.

\begin{definition}\label{auxiset}Let $G(t)$ be a compact, nonempty and nonconvex level set. Under assumptions (A) and (D), the set $G^+(t)\subset \mathbb{R}^d$ is defined as the level set with threshold equal to $t+2D_n$. That is, $G^+(t)=G(t+2D_n)$.
\end{definition}

\begin{definition}\label{auximuestra}Let $G(t)$ be a compact, nonempty and nonconvex level set. Under assumptions (A) and (D), let $\mathcal{X}_n$ be a random sample generated from a distribution with density function $f$ and let $G^+(t)\subset \mathbb{R}^d$ be the level set established in Definition \ref{auxiset}. The set $\mathcal{X}_n^{G^+}$ is defined by $\mathcal{X}_n\cap G^+(t)$. Therefore, it can be written as $\mathcal{X}_n^{G^+}=\{X_i\in\mathcal{X}_n:f(X_i)\geq t+2D_n\}$.
\end{definition}

A new class of sets is presented in Definition \ref{waltherG}. This family was already considered in \cite{r2}, see Section 2.

\begin{definition}\label{waltherG}Let $A\subset \mathbb{R}^d$ be a set and $\gamma>0$. Then, $\mathcal{G}_{A}(\gamma)$ denotes all compact sets $B$ that verify ($R_{\gamma}^\gamma$) satisfying $B\subset A$.
\end{definition}


Next, it will be proved that $G^+(t)\in\mathcal{G}_{G(l)}(r_\nu)$ for $n$ large enough and $r_\nu>0$, see Lemma \ref{debemosescribirlo55} for details about the positive constant $r_\nu$.

\begin{lemma}\label{debemosescribirlo55}Let $G(t)$ be a compact, nonempty and nonconvex level set. Let $G^+(t)$ be the set established in Definition \ref{auxiset}. Under assumptions (A), (D) and (K), let $\hat{r}_0(t)$ be as established in Definition \ref{jejejeje}, let $\nu\in(0,1)$ be a fixed number and $r_n(t)=\nu\hat{r}_0(t)$. Then, there exists $0<r_\nu<m/k$ such that
	$$\mathbb{P}(r_n(t)>r_\nu, \mbox{ }\forall t\in [l,u],\mbox{ eventually})=1.$$Further, for $n$ large enough,
	$$G^+(t)\in\mathcal{G}_{G(t)}(r_\nu), \forall t\in [l,u] $$and, therefore,
	$$G^+(t)\in\mathcal{G}_{G(l)}(r_\nu), \forall t\in [l,u].$$
	for $\mathcal{G}_{G(t)}(r_\nu)$, $\mathcal{G}_{G(l)}(r_\nu)$ and $G^+(t)$ established in Definitions \ref{waltherG} and \ref{auxiset}, respectively.
	
\end{lemma}

\begin{proof}	For all $t\in[l,u]$, $G(t)$ satisfied the inside and outside rolling condition for $r^*=m/k>0$. Since
	$$r_0(t)=\sup\{\gamma>0: C_\gamma(G(t))=G(t)\},$$it is satisfied that $$0<r^*\leq r_0(t),\mbox{ }\forall t\in[l,u].$$
	Therefore, for $\nu\in(0,1)$ fixed previously,$$0<\nu r^*\leq \nu r_0(t) ,\mbox{ } \forall t\in[l,u].$$Let $r_\nu>0$ be a number verifying that
	$$0<r_\nu<\nu r^*\leq \nu r_0(t) ,\mbox{ } \forall t\in[l,u].$$
	Since $r_\nu<r^*<r_0(t)$, $G(t)$ satisfies the outside and inside rolling property for $r_\nu$ and for all $t\in[l,u]$.
	Let $0<\epsilon< \nu r^*-r_\nu$, according Theorem \ref{co}, it is easy to prove that $\mathbb{P}(r_n(t)>r_\nu, \forall t\in[l,u], \mbox{ eventually})=1$ taking into account that $ r_n(t)$ converges uniformly to $\nu r_0(t)$, almost surely. \\In addition, if Theorem 2 in \cite{r2} is applied with $\overline{l}=l-\zeta/2$, $\overline{u}=u+\zeta/2$ and $\overline{\zeta}=\zeta/2$, we can ensure that
	$$\exists n_0\in\mathbb{N}\mbox{ such that } r_0(t+2D_n)\geq m/k,\mbox{ }\forall t\in [l,u]\mbox{ and }\forall n\geq n_0 .$$
	Since $0<r_\nu<m/k$,
	$$r_0(t+2D_n)\geq m/k>r_\nu,\mbox{ }\forall t\in [l,u]\mbox{ and }\forall n\geq n_0.$$
	Then, since $G^+(t)=G(t+2D_n)\subset G(t)$ for all $n$,
	\[G^+(t) \in\mathcal{G}_{G(t)}(r_\nu),\mbox{ }\forall t\in [l,u]\mbox{ and }\forall n\geq n_0.\qedhere\]
\end{proof}

In Lemma \ref{dddd}, it will proved that, given the threshold $t$, the set $\mathcal{X}_{n}^{G^+} $ is eventually contained in $\mathcal{X}_{n}^+(t)$.

\begin{lemma}\label{dddd}Let $G(t)$ be a compact, nonempty and nonconvex level set. Under assumptions (A), (D) and (K), let $\mathcal{X}_n$ be a random sample generated from a distribution with density function $f$, let $\mathcal{X}_n^+(t)$ be as established in Definition \ref{jejejeje} and let $\mathcal{X}_{n}^{G^+}$ be the subsample defined in Definition \ref{auximuestra}. Then,
	$$\mathbb{P}(\mathcal{X}_{n}^{G^+}\subset \mathcal{X}_{n}^+(t), \mbox{ }\forall t\in[l,u],   \mbox{ eventually})=1.$$
\end{lemma}

\begin{proof}Let $X_i\in \mathcal{X}_{n}^{G^+}$. Therefore, $f(X_i)\geq t+2D_n$. According to Proposition \textcolor{blue}{A.1} in Appendix A,
	$$\sup_{C}|f_n-f|=O\left(  \left(  \frac{\log{n}}{n}\right)^{p/(d+2p)}\right), \mbox{ almost surely.}$$where $C\subset U$ is under conditions of Proposition \textcolor{blue}{A.1} in Appendix A. So, there exists $N>0$ such that, with probability one and for $n$ large enough,
	\begin{equation}\label{gt}\sup_{C}|f_n-f|\leq N  \left(  \frac{\log{n}}{n}\right)^{p/(d+2p)} .
	\end{equation}Two cases are considered: $X_i$ belongs to $C$ or $X_i$ does not belong to $C$.
	\begin{enumerate}
		\item  Let $X_i\in C $. According to (\ref{gt}), if $M\geq N$,
		$$|f_n(X_i)-f(X_i)|\leq D_n .$$Therefore,
		$$f_n(X_i)\geq f(X_i)-D_n\geq t +2D_n -D_n=t+D_n$$and, hence, $X_i\in\mathcal{X}_{n}^+(t)$.
		
		\item If $X_i\notin C $ then $X_i\in G(u)\cap C^c$ since $X_i\in G(l)$ and $G(l)\setminus \interior(G(u))\subset C$. According to Proposition \textcolor{blue}{A.2} in Appendix A, with probability one and for $n$ large enough, $f_n(X_i)\geq u+ v/2$ for some $v>0$. In addition, since $D_n$ converges to zero, for $n$ large enough, $2D_n< v/2$. Then, with probability one and for $n$ large enough,
		\[f_n(X_i)\geq u+ \frac{v}{2}\geq t+\frac{v}{2}\geq t+D_n. \qedhere\]
\end{enumerate}\end{proof}

According to Lemma \ref{dddd}, with probability one and for $n$ large enough, it is verified that $C_{r_n(t)}(\mathcal{X}_{n}^{G^+})\subset C_{r_n(t)}(\mathcal{X}_{n}^+(t))\subset G(t)$. That is, $\mathcal{X}_{n}^+(t)$ is at least as good as $\mathcal{X}_{n}^{G^+}$ in order to estimate $G(t)$. Remember that  $\mathcal{X}_{n}^{G^+}$ is constructed from $f$. It does not depend on the kernel estimator $f_n$. In addition, $\mathcal{X}_{n}^{G^+}$ would be the ideal sample for estimating $G^+(t)$. Theorem \ref{principal} uses these ideas for obtaining the convergence rates of the level set estimator proposed.

Let $r_\nu$ be a positive constant under the conditions in Lemma \ref{debemosescribirlo55} and let $r>0$ such that $0<r<r_\nu$. Let $\epsilon_n=\left(\frac{C\log{n}}{n}\right)^{\frac{2}{d+1}}$ where $C>0$ denotes a large enough constant to be established later. Since $\lim_{n\rightarrow\infty}\epsilon_n=0$,
$$\exists n_0\in\mathbb{N}\mbox{ such that }0<\epsilon_n<\min\left\{\frac{r}{4},1\right\},\mbox{ }\forall n\geq n_0.$$
According to Proposition \textcolor{blue}{A.4} in Appendix A, since $f\geq l>0$,
$$\mathbb{P}(A\oplus B_{r-3\epsilon_n}[0]\nsubset \left[\left(A\cap\mathcal{X}_n\right)\oplus B_r[0]\right]\mbox{ for some }A\in \mathcal{G}_{G(l)}(r_\nu))$$
$$\leq D(\epsilon_n, G(l)\oplus B_r[0])D\left(\frac{\epsilon_n}{10r},S^{d-1}\right)\exp\left\{-nal(r-2\epsilon_n)^\frac{d-1}{2}(\epsilon_n/2)^{\frac{d+1}{2}}\right\},$$where $D(\epsilon,B)=\max\{card\mbox{ } V:V\subset B,\mbox{ }|x-y|>\epsilon\mbox{ for different }x,y\in V\}$, $S^{d-1}$ denotes the unit sphere in $\mathbb{R}^d$ and $a$ is a dimensional constant. Therefore, if $n\geq n_0$ then
$r-2\epsilon_n\geq r/2$ and
$$\mathbb{P}(A\oplus B_{r-3\epsilon_n}[0]\nsubset \left[\left(A\cap\mathcal{X}_n\right)\oplus B_r[0]\right]\mbox{ for some }A\in \mathcal{G}_{G(l)}(r_\nu))$$
$$\leq Q \epsilon_n^{-d} \epsilon_n^{-(d-1)}\exp\left\{-nal\left(\frac{r}{2}\right)^{\frac{d-1}{2}}\left(\frac{C\log{n}}{2^{(d+1)/2}n}\right)\right\}=Q\epsilon_n^{(-2d+1)}\exp\{-W\log{n}\}$$with $Q$ is a constant depending on $r$ and the dimension $d$ and $W=\frac{alr^{d-1}C}{2^d}.$ Therefore, if $n\geq n_0$,
$$\mathbb{P}(A\oplus(r-3\epsilon_n)B_1[0]\nsubset \left[\left(A\cap\mathcal{X}_n\right)\oplus B_r[0]\right]\mbox{ for some }A\in \mathcal{G}_{G(l)}(r_\nu))$$
$$\leq Q\epsilon_n^{-(2d-1)}n^{-W}= Q^{'}\left(\frac{n}{\log{n}}\right)^{\frac{2(2d-1)}{d+1}}n^{-W}$$where $Q^{'}=QC^{\frac{2(2d-1)}{d+1}}$. If $W>\frac{2(2d-1) }{d+1}$ it is verified that
$$\sum_{i=1}^\infty \left(\frac{n}{\log{n}}\right)^{\frac{2(2d-1) }{d+1}}n^{-W}<\infty.$$So,
$$\mathbb{P}(A\oplus B_{r-3\epsilon_n}[0]\nsubset \left[\left(A\cap\mathcal{X}_n\right)\oplus B_r[0]\right]\mbox{ for some }A\in \mathcal{G}_{G(l)}(r_\nu),\mbox{ infinitely often})=0.$$
Then, with probability one and for $n$ large enough,
$$ A\oplus B_{r-3\epsilon_n}[0]\subset  \left(A\cap\mathcal{X}_n\right)\oplus B_r[0] ,\mbox{ }\forall A\in\mathcal{G}_{G(l)}(r_\nu).$$According to Lemma \ref{debemosescribirlo55}, $G^+(t)\in \mathcal{G}_{G(l)}(r_\nu)$ for all $t\in[l.u]$. Therefore, with probability one and for $n$ large enough,
$$(G^+(t)\oplus B_{r-3\epsilon_n}[0])\ominus B_r[0]\subset (\mathcal{X}_n^{G^+}\oplus B_r[0])\ominus B_r[0]       =C_r(\mathcal{X}_n^{G^+}),\mbox{ }\forall t\in [l,u].$$
Since $G^+(t)$ is $(r-3\epsilon_n)-$convex because $r-3\epsilon_n\leq r_\nu$, it is satisfied that
$$(G^+(t)\oplus B_{r-3\epsilon_n}[0])\ominus B_r[0]=$$
$$=(G^+(t)\oplus  B_{r-3\epsilon_n}[0])\ominus( B_{r-3\epsilon_n}[0]\oplus B_{3\epsilon_n}[0])$$
$$=(G^+(t)\oplus  B_{r-3\epsilon_n}[0]\ominus B_{r-3\epsilon_n}[0])\ominus  B_{3\epsilon_n}[0]=G^+(t)\ominus  B_{3\epsilon_n}[0],\mbox{ }\forall t\in [l,u].$$
Therefore, since $r_\nu>r>0$, with probability one and for $n$ large enough,
$$G^+(t)\ominus  B_{3\epsilon_n}[0]\subset C_{r}(\mathcal{X}_n^{G^+})\subset   C_{r_\nu}(\mathcal{X}_n^{G^+}) \mbox{ }\forall t\in [l,u].$$According to Lemma \ref{debemosescribirlo55}, with probability one, for $n$ large enough,
$$r_n(t)\geq r_\nu,\mbox{ }\forall t\in [l,u].$$
Then, with probability one and for $n$ large,
$$G^+(t)\ominus  B_{3\epsilon_n}[0]\subset C_{r_n(t)}(\mathcal{X}_n^{G^+}),\mbox{ }\forall t\in [l,u].$$
According to Lemma \ref{dddd} and Proposition \ref{mostramaisenGlambda}, $\mathcal{X}_n^{G^+}\subset\mathcal{X}_n^+(t)\subset G(t) $ and $r_n(t)\leq r_0(t)$. Then, it is verified, with probability one, for $n$ large enough,
\begin{equation}\label{900}
	G^+(t)\ominus   B_{3\epsilon_n}[0]\subset C_{r_n(t)}(\mathcal{X}_n^+(t))\subset C_{r_0(t)}(G(t))=G(t),\mbox{ }\forall t\in [l,u].
\end{equation}
Using (\ref{900}), with probability one and for $n$ large enough,
$$d_H(C_{r_n(t)}(\mathcal{X}_n^+(t)),G(t))\leq d_H(G^+(t)\ominus B_{3\epsilon_n}[0],G(t)).$$By the triangle inequality,
\begin{equation}\label{juioplr}
	d_H(C_{r_n(t)}(\mathcal{X}_n^+(t)),G(t)) \leq d_H(G^+(t),G(t))+d_H(G^+(t),G^+(t)\ominus   B_{3\epsilon_n}[0]),\mbox{ }\forall t\in [l,u].
\end{equation}
Since $\lim_{n\rightarrow \infty}D_n=0$, for $n$ large, $2D_n<\min\left\{(m/2)c,\frac{\zeta}{2}\right\}$ and, according to Proposition \textcolor{blue}{A.3} (b) in Appendix A, for $n$ large,
$$G(t)\subset  G^+(t)\oplus B_{\frac{4}{m}D_n}[0],\mbox{ }\forall t\in [l,u].$$
Since $G^+(t)\subset G(t)$, $d_H(G^+(t),G(t))=O(D_n)$. On the other hand,
$d_H(G^+(t),G^+(t)\ominus  B_{3\epsilon_n}[0])=O(\epsilon_n) \mbox{ for }t\in[l,u]$. As consequence, using (\ref{juioplr}),
\[ d_H(C_{r_n(t)}(\mathcal{X}_n^+(t)),G(t))=O(\max\{ D_n,\epsilon_n \}), \mbox{ almost surely.} \]

The statements of the theorem also hold when Lebesgue is used instead of Hausdorff distance. See Remark 1 in \cite{r2}.

$\vspace{-.24cm} $

\textbf{Proof of Theorem \ref{principal2}:}\vspace{-.23cm}\\

Let $a_n=(4/m)D_n+\epsilon_n$ where $\epsilon_n$ denotes a sequence of the order $\left(\frac{\log{n}}{n}\right)^{\frac{2}{d+1}}$ established in the proof of Theorem \ref{principal}. Next, it will be proved that, with probability one and for $n$ large enough,
\begin{equation}\label{jip6}
	G(t)\ominus a_nB_1[0]\subset C_{r_n(t)}(\mathcal{X}_n^+(t))\subset G(t)\oplus a_nB_1[0],\mbox{ }\forall t\in[l,u].
\end{equation}According to equation (\ref{900}), with probability one and for $n$ large enough,
$$ G^+(t)\ominus 3\epsilon_n B_1[0]\subset C_{r_n(t)}(\mathcal{X}_n^+(t))\subset G(t),\mbox{ } \forall t\in[l,u].
$$Let $c$ be the constant introduced in Proposition \textcolor{blue}{A.3} (b) in Appendix A. Since $\lim_{n\rightarrow\infty}D_n=0$, for $n$ large enough, $ 2D_n\leq mc/2$.\\
Then, Proposition \textcolor{blue}{A.3} (b) ensures that
$$G^+(t) \ominus\frac{4}{m}D_nB_1[0]\subset G(t)\subset G^+(t) \oplus\frac{4}{m}D_nB_1[0]$$for $n$ large enough.
Therefore, for all $G^+(t)\in\mathcal{G}_c(r_\nu)$, $t\in[l,u]$ and $n$ large enough,
$$ G(t)\ominus a_n B_1[0] \subset G^+(t)\ominus 3\epsilon_n B_1[0]\subset C_{r_n(t)}(\mathcal{X}_n^+(t))\subset G(t)\subset G(t)\oplus a_n B_1[0].$$
Next, it will be proved that
\begin{equation}\label{jip}
	\sup_{t\in[l,u]}|\mathbb{P}_n(G(t))-\mathbb{P}(G(t))|\leq a_n.
\end{equation}According to Proposition 1 in \cite{r2}, for $r>0$ and a compact $C$ it is verified
\begin{equation*}
	\sup_{A\in \mathcal{G}_C(r)}|\mathbb{P}_n(A)-\mathbb{P}(A)| =
	\begin{cases}
		O\left(\frac{1}{n^{2/(d+1)}}\right)\mbox{ }a.s. & \text{if $d>3$}\vspace{1.6mm}\\
		O\left(\frac{\log{n}}{\sqrt{n}}\right)\mbox{ }a.s. & \text{if $d=1,2,3$.}
	\end{cases}
\end{equation*}Theorem 2 in \cite{r2} ensures that $G(t)\in\mathcal{G}_{G(l)}(m/k)$, for all $t\in[l,u]$. \\		
Following the proof of Theorem 3 in \cite{r2} (see equation (23), page 2295), (\ref{jip6}) and (\ref{jip}) guarantee that, for $n$ large and certain constant $M$,
\begin{equation}\label{jip9}
	\sup_{\tau\in[\underline{\tau},\overline{\tau}]}|\hat{f}_\tau-f_\tau|\leq M a_n.
\end{equation}Furthermore, since $\lim_{n\rightarrow \infty}a_n=0$, it is verified that
$Ma_n\leq (m/2)c$ for $n$ large enough.\\
In addition, equation (\ref{jip9}) and Proposition \textcolor{blue}{A.3} (b) allows to ensure that, for all $\tau\in[\underline{\tau},\overline{\tau}]$ and for $n$ large enough (assuming, for instance, $\hat{f}_{\tau}>f_{\tau}$; similarly if $ \hat{f}_\tau<f_{\tau}$),
$$L(\tau) \ominus \frac{2}{m}Ma_n B_1[0]\subset G(\hat{f}_\tau)\subset L(\tau) \oplus \frac{2}{m}Ma_n B_1[0].$$Taking also into account equation (\ref{jip6}), with probability one and for $n$ large enough,
\[L(\tau)\ominus \left(\frac{2}{m}Ma_n+a_n\right)B_1[0]\subset C_{r_n(\hat{f}_\tau)}(\mathcal{X}_n^+(\hat{f}_\tau))\subset L(\tau)\oplus\left (\frac{2}{m}Ma_n+a_n\right)B_1[0].\] Since $a_n=(4/m)D_n+\epsilon_n$,
\[ d_H(C_{r_n(\hat{f}_\tau)}(\mathcal{X}_n^+(\hat{f}_\tau))),L(\tau))=O(\max\{ D_n,\epsilon_n \}), \mbox{ almost surely.}  \]
As before, the statements of the theorem also hold when Lebesgue is used instead of Hausdorff distance. See Remark 1 in \cite{r2}.

$ $\\

\appendix
\textbf{Supplementary material for Minimax Hausdorff estimation of density level sets.} Appendix A contains some mathematical results in \cite{r2}. They are useful for establishing some proofs in this work. In Appendix B, some auxiliary results are proved. They are fundamental to prove the consistency of the estimator $\hat{r}_0(t)$.

\section{Walther's mathematical results}\label{app}

Many proofs in previous sections take into account mathematical aspects considered in \cite{r2}. Next, we will summarize these theoretical results.  In particular, Proposition \ref{enPruebaTeorema3Walther} can be obtained directly from proof of Theorem 3 in \cite{r2}. It guarantees the existence of a compact set  $C$ where the convergence rate for the density kernel estimator is established.



\begin{proposition}\label{enPruebaTeorema3Walther}Under assumptions (A) and (K), there exist $\upsilon>0$ and a compact set $C$ verifying that $$G(l)\setminus \interior (G(u))\oplus  B_{\upsilon}[0]\subset U$$and$$G(l)\setminus \interior (G(u))\oplus B_{\frac{\upsilon}{2}}[0]\subset C.$$See assumption A for details on $U$. In addition,
	\begin{equation*}
		\sup_{C}|f_n-f| =O\left(  \left(  \frac{\log{n}}{n}\right)^{p/(d+2p)}\right),\mbox{ almost surely}.
	\end{equation*}
	\begin{figure}[h!]\centering
		\begin{pspicture}(-2,-1.95)(10,6.5)
		
		\rput(-1.58,-1.){\scalebox{1.3}[1.4]{\psccurve[showpoints=false,fillstyle=solid,fillcolor=white,linecolor=black,linewidth=0.15mm,linearc=3,linestyle=dashed,dash=2pt 2pt](1.5,0)(1.9,2.5)(1.5,5)(4,4.5)(7,4)(6.2,-0.01)(5,0.6)}}
		\psccurve[showpoints=false,fillstyle=solid,fillcolor=white,linecolor=black,linewidth=0.25mm,linearc=3](1.5,0)(1.9,2.5)(1.5,5)(4,4.35)(7,4)(6.2,.3)(5,1.5)

		\psccurve[showpoints=false,fillstyle=solid,fillcolor=white,linecolor=gray,linewidth=0.25mm,linearc=3](5.8,1.3)(6.6,3.5)(6.15,0.9)
		
		\rput(1.93,0.61){\scalebox{0.7}[0.75]{\psccurve[showpoints=false,fillstyle=solid,fillcolor=white,linecolor=black,linewidth=0.25mm,linearc=3,linestyle=dashed,dash=3pt 2pt](5.8,1.3)(6.6,3.5)(6.15,0.9) }}
		\psccurve[showpoints=false,fillstyle=solid,fillcolor=white,linecolor=gray,linewidth=0.25mm,linearc=3](2,0.8)(2.9,2)(2.55,0.8)
		
		\rput(0.9,0.5){\scalebox{0.65}[0.65]{\psccurve[showpoints=false,fillstyle=solid,fillcolor=white,linecolor=black,linewidth=0.25mm,linearc=3,linestyle=dashed,dash=3pt 2pt](2,0.8)(2.9,2)(2.55,0.8)  }}
		
		\rput(5.5,3){$G(u)$}
		\rput(7.8,5.00){$U$}
		\rput(0.8,0.3){$G(l)$}
		\end{pspicture}\vspace{-.3cm}\caption{Elements in Proposition \ref{enPruebaTeorema3Walther}. $G(u)$ in gray, $G(l)$ in black and the open set $U$ in dashed line.}\label{oookffghhffffffffffffffffffffffhg1o}
	\end{figure}
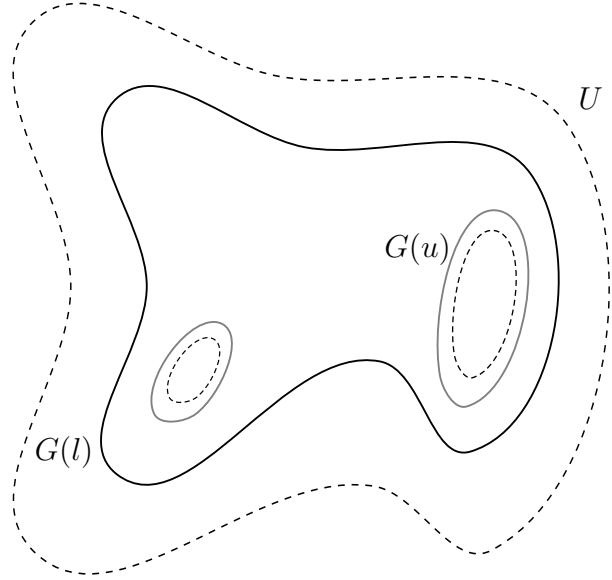
\end{proposition}


Proposition \ref{enPruebaTeorema3Walther3} corresponds to equation (15) in \cite{r2}. The behavior of the kernel density estimator is studied in the complement of the compact set $C$.

\begin{proposition}\label{enPruebaTeorema3Walther3}Let $C$ be the compact set in Proposition \ref{enPruebaTeorema3Walther}. Under assumptions (A) and (K), there exists $w>0$ verifying that
	\begin{equation*}
		\mathbb{P}\left( \inf_{G(u)\cap C^c} f_n(x)>u+\frac{w}{2},\mbox{ eventually}\right)=1
	\end{equation*}and\begin{equation*}
		\mathbb{P}\left( \sup_{G(l)^c \cap C^c } f_n(x)<l-\frac{w}{2},\mbox{ eventually} \right)=1.
	\end{equation*}
	
\end{proposition}

Proposition \ref{enPruebaTeorema3Walther4} corresponds to Lemma 2 (a) and (b) in \cite{r2}. It establishes interesting relationships between level sets with close enough thresholds. Lemma 2 (b) in \cite{r2} was slightly modified but the result remains true if the interval $[l,u]$ is replaced by $[l^{-},u^{+}]$ for certain $l^{-}<l$ and $u<u^{+}$.

\begin{proposition}\label{enPruebaTeorema3Walther4}Under assumption (A), there exists a constant $c>0$ such that
	\begin{itemize}
		\item[(a)]For all $t\in [l,u]$ and all $S>0$ such that $S\leq c$,
		$$ \sup_{G(t)^c\ominus S B_1[0]}f(x)<t-b_t$$for a certain $b_t>0$.
		\item[(b)] Let $l^{-}=l-\zeta/2$ and $u^{+}=u+\zeta/2$. If $t_1$ and $t_2$ are such that $l^{-}\leq t_1<t_2\leq u^{+}$ and $t_2-t_1\leq (m/2)c$ then
		\begin{equation*}
			G(t_1)\ominus B_{\frac{2}{m}(t_2-t_1)}[0]\subset G(t_2)\subset G(t_1)\ominus B_{\frac{1}{2m}(t_2-t_1)}[0]\end{equation*}and
		\begin{equation*}
			G(t_2)\oplus B_{\frac{1}{2m}(t_2-t_1)}[0]\subset G(t_1)\subset G(t_2)\oplus B_{\frac{2}{m}(t_2-t_1)}[0].\end{equation*}See assumption A for details on $\zeta$.
	\end{itemize}
\end{proposition}

Finally, Proposition \ref{enPruebaTeorema3Walther5} is presented. It corresponds to Lemma 3 in \cite{r2}.

\begin{proposition}\label{enPruebaTeorema3Walther5}Let $T\subset\mathbb{R}^d$ be a compact set, $r > 0$ and let $\mathcal{X}_n$ be a i.i.d. sample generated from a distribution with density function $f$. Let $ \mathcal{G}_T(r)$ be the family of sets defined in Definition \textcolor{blue}{6.3}.
	\begin{enumerate}
		\item If $f \geq b > 0$ on $A\in \mathcal{G}_T(r)$ and $0 < \epsilon< \min\{\overline{r}/2 , r\}$ then
		$$
		\mathbb{P}\left(A\oplus B_{\overline{r}-2\epsilon}[0]\nsubset (A\cap\mathcal{X}_n)\oplus B_{\overline{r}}[0]\right)$$
		$$\leq
		D\left(\epsilon,A\oplus B_{\overline{r}}[0]\right) \exp{ \left( -nab \min\{\overline{r}-\epsilon,r\}^{(d-1)/2}\epsilon^{(d+1)/2} \right) }.
		$$where$$D\left(\epsilon,A\oplus B_{\overline{r}}[0]\right)= \max\{card\mbox{ } V:V\subset A\oplus B_{\overline{r}}[0],\mbox{ }|x-y|>\epsilon\mbox{ for different }x,y\in V\}$$and $a$ is a dimensional constant.\vspace{.3mm}
		\item Further, if $f \geq b > 0$ on $T$, $0 < \epsilon< \min\{\overline{r}/3 , 1\}$ and $r\geq \overline{r}-2\epsilon$ then
		$$
		\mathbb{P}\left(A\oplus B_{\overline{r}-3\epsilon}[0]\nsubset (A\cap\mathcal{X}_n)\oplus B_{\overline{r}}[0] \mbox{ for some }A\in \mathcal{G}_T(r)\right)$$
		$$\leq
		D\left(\epsilon,T\oplus B_{\overline{r}}[0]\right)D\left(\frac{\epsilon}{10\overline{r}},S^{d-1}\right)\exp{\left(-nab(\overline{r}-2\epsilon)^{(d-1)/2}(\epsilon/2)^{(d+1)/2}\right)}
		$$where $S^{d-1}$ denotes the unit sphere.
	\end{enumerate}
	
\end{proposition}

\section{Auxiliary mathematical results}\label{app2}

Next, some useful theoretical results to prove the consistency of the estimator $\hat{r}_0(t)$ are exposed. All of them are focused on finding or constructing balls which satisfying certain properties. This mathematical objects will play a fundamental role to get the consistency results. In a first step, an open ball which does not intersect $G(t)$ can be found in the $r-$convex hull of $G(t)$ for all $r>r_0(t)$, see Proposition \ref{Alambda}.

\begin{proposition}\label{Alambda}Let $G(t)$ be the compact, nonempty and nonconvex level set. Under (A), let $r_0(t) $ be as established in Definition  \textcolor{blue}{2.1} and $f$, the density function. Then, for all $r>r_0(t)$ there exists an open ball $A_t$ with radius $\rho_t$ centering at $c_t$ contained in the compact $C$ established in Proposition \ref{enPruebaTeorema3Walther} verifying
	\begin{equation}\label{distannnn}
		A_t\subset C_{r}(G(t)) \mbox{ and }	d(z,G(t))\geq  3\rho_{t},\mbox{ for all } z\in A_t.
	\end{equation}Therefore,
	\begin{equation}\label{estrela}
		A_t\cap B_{2\rho_t}(G(t))=\emptyset.
	\end{equation}
	In addition, $f(x)>l-\zeta$, for all $x\in A_t$.
\end{proposition}

\begin{figure*}[h!]
	\begin{pspicture}(-1.7,-1.9)(10,5.3)
	\psccurve[showpoints=false,fillstyle=solid,fillcolor=white,linecolor=gray,linewidth=.75mm,linearc=3](1.5,0)(1,2.5)(1.5,5)(4,4.85)(7,4)(6.2,-0.05)(5,1.5)
	\psccurve[showpoints=false,fillstyle=solid,fillcolor=white,linecolor=white,linewidth=.85mm,linearc=3](1.5,0)(6.2,-0.05)(5,1.5)
	\psarc[showpoints=false,fillstyle=solid,fillcolor=white,linecolor=gray,linewidth=0.45mm,linearc=3](6.108,.7351) {.8}{264}{280}
	\psarc[showpoints=false,fillstyle=solid,fillcolor=red,linecolor=gray,linewidth=0.45mm,linearc=3](1.87,0.805) {.9}{240}{281}
	\psarc[showpoints=false,fillstyle=solid,fillcolor=white,linecolor=gray,linewidth=0.45mm,linearc=3](4.1,-9.867) {10}{78.9}{102}
	\psccurve[showpoints=false,fillstyle=solid,fillcolor=white,linecolor=black,linewidth=0.3mm,linearc=3](1.5,0)(1,2.5)(1.5,5)(4,4.85)(7,4)(6.2,-0.05)(5,1.5)
	
	\pscircle[linearc=0.25,linecolor=black,linewidth=0.2mm,linestyle=dashed,dash=3pt 3pt](4.8,1.53){1.}
	\pscircle[linearc=0.25,linecolor=gray,linewidth=0.2mm,linestyle=solid,dash=3pt 2pt,fillstyle=crosshatch*,fillcolor=gray,hatchcolor=white,hatchwidth=1.2pt,hatchsep=.5pt,hatchangle=0](4.8,1.03){.12}
	\rput(4.8,1.685){\small{$x_t$}}
	\rput(1.7,4){$G(t)$}
	\rput(4.8,2.83){\small{$B_{\epsilon^{*}_{t}}(x_t)$}}
	\rput(5.15,0.9){\small{$A_t$}}
	\psdots*[dotsize=3pt](4.8,1.53)
	\rput(4.75,-.8){$C_{r}(G(t))$}
	\end{pspicture}\vspace{-1cm}\caption{Elements of proof in Proposition \ref{Alambda}. }\label{oosssokkkdkd}
\end{figure*}
\begin{proof}

	Since $r>r_0(t)$, there exists $\epsilon>0$ such that ${r}=r_0(t)+\epsilon$. Taking into account Lemma 8.3 in \cite{r00000}, it is easy to prove that there exists $x_{t}\in\partial G(t)\cap \interior(C_{r}(G(t)))$. Therefore,
	$$\exists \epsilon_{t}>0\mbox{ such that }B_{\epsilon_{t}}(x_{t})\subset C_{r}(G(t)).$$
	

	Proposition \ref{enPruebaTeorema3Walther} guarantees that there exists $\upsilon>0$ and a compact set $C$ such that
	$$(G(l)\setminus \interior (G(u))\oplus B_{\frac{\upsilon}{2}}[0]\subset C.$$
	Since $x_t\in\partial G(t)$, it is verifies that $x_t\in G(l)\setminus\interior(G(u))$. Therefore, $$B_{\upsilon/2}(x_t)\subset (G(l)\setminus \interior(G(u)))\oplus\frac{\upsilon}{2} B_{1}[0]\subset C.$$In addition, if $\upsilon$ is small enough, we can assume without loss of generality that $$(G(l)\setminus \interior (G(u))\oplus \frac{\upsilon}{2} B_1[0] \subset G(l-\zeta/2)\setminus G(u+\zeta/2). $$
	
	Let $\epsilon_t^*=\min\{\epsilon_t, \upsilon/2\}$. Then, we define $A_t=B_{\rho_t}(c_t)$ where $\rho_t=\epsilon_t^*/8$, $c_t=x_t+\frac{\epsilon_{t}^*}{2}\eta(x_{t})$ and $\eta(x_t)$ denotes the outward pointing unit normal vector at $x_t$. Under (A), the existence of $\eta(x_t)$ is guaranteed, see proof of Proposition 2.2 in \cite{r00000}.

By construction, $A_t\subset B_{\epsilon_t^*}(x_t)\subset B_{\upsilon/2}(x_t)$ then it is verified that $A_t\subset C $ and $A_t\subset G(l-\zeta/2)\setminus G(u+\zeta/2)$. Then, $f(x)\geq l-\zeta/2>l-\zeta$ for all $x\in A_t$. In addition, given $z\in A_t$, it is verified that
$$	d(z,G(t))\geq d\left(x_t+\left(\frac{1}{2}-\frac{1}{8}\right)\epsilon_{t}^*\eta(x_t),x_t\right)=\frac{3}{8}\epsilon_{t}^*=3\rho_t>0.	 \qedhere$$\end{proof}

In Lemma \ref{contenidos}, if $r>r_0(t)$ then we find an open ball in $C_r(\mathcal{X}_n^+(t))$ that does not intersect $G(t)$ with probability one and for $n$ large enough. As consequence, it will be contained in $C_r( \mathcal{X}_n^+(\overline{t})$ for all $\overline{t}<t$.

\begin{lemma}\label{contenidos}Let $G(t)$ be a compact, nonempty and nonconvex level set. Under assumptions (A), (D) and (K), let $\mathcal{X}_n$ be a random sample generated by the density function $f$, let $r_0(t)$ be as established in Definition \textcolor{blue}{2.1} and let $\mathcal{X}_n^+(t)$ be as established in Definition \textcolor{blue}{2.2}. For $r>r_0(t)$ and $\overline{t}\in[l,u]$ verifying  $\overline{t}<t$, there exists an open ball $B_t$ of radius $\gamma_t\leq c/2$ such that $B_t\subset C$ where $C$ is the compact set established in Proposition \ref{enPruebaTeorema3Walther},$$B_t\cap G(t)=\emptyset,\mbox{ }B_t\subset C_{{r}}(G(t))$$and,
$$\mathbb{P}(B_t\subset C_{{ r}}(\mathcal{X}_n^+(\overline{t})), \mbox{ eventually})=1.$$In addition, it is verified that $f(x)>l-\zeta$ for all $x\in B_t$ and $d(z,G(t))\geq 6\gamma_t$, for all $z\in B_t$. See Proposition \ref{enPruebaTeorema3Walther4} for details about constant $c$.
\end{lemma}
\begin{figure*}[h!]
\begin{pspicture}(-1.7,-1.9)(10,7)
\rput(-1.2,-1.23){\scalebox{1.38}[1.5]{\psccurve[showpoints=false,fillstyle=solid,fillcolor=white,linecolor=black,linewidth=0.3mm,linearc=3](1.5,0)(1,2.5)(1.5,5)(4,4.85)(7,4)(6.2,-0.05)(4,0.1)}}
\psccurve[showpoints=false,fillstyle=solid,fillcolor=white,linecolor=gray,linewidth=.75mm,linearc=3](1.5,0)(1,2.5)(1.5,5)(4,4.85)(7,4)(6.2,-0.05)(5,1.5)
\psccurve[showpoints=false,fillstyle=solid,fillcolor=white,linecolor=white,linewidth=.85mm,linearc=3](1.5,0)(6.2,-0.05)(5,1.5)
\psarc[showpoints=false,fillstyle=solid,fillcolor=white,linecolor=gray,linewidth=0.45mm,linearc=3](6.108,.7351) {.8}{264}{280}
\psarc[showpoints=false,fillstyle=solid,fillcolor=red,linecolor=gray,linewidth=0.45mm,linearc=3](1.87,0.805) {.9}{240}{281}
\psarc[showpoints=false,fillstyle=solid,fillcolor=white,linecolor=gray,linewidth=0.45mm,linearc=3](4.1,-9.867) {10}{78.9}{102}
\psccurve[showpoints=false,fillstyle=solid,fillcolor=white,linecolor=black,linewidth=0.3mm,linearc=3](1.5,0)(1,2.5)(1.5,5)(4,4.85)(7,4)(6.2,-0.05)(5,1.5)

\pscircle[linearc=0.25,linecolor=black,linewidth=0.2mm,linestyle=solid,dash=3pt 2pt](4.8,1.03){.4}
\pscircle[linearc=0.25,linecolor=gray,linewidth=0.2mm,linestyle=solid,dash=1pt 1pt,fillstyle=crosshatch*,fillcolor=gray,hatchcolor=white,hatchwidth=1.2pt,hatchsep=.5pt,hatchangle=0](4.8,1.03){.2}
\rput(1.7,4){$G(t)$}
\rput(4,0.9){\small{$A^*_t$}}
\rput(4.75,-.4){$C_{r^*}(G(t))$}
\rput(9.6,3){$G(\overline{t})$}
\end{pspicture}\vspace{-.2cm}\caption{Elements of proof in Lemma \ref{contenidos}. $B_t$ in gray color. }\label{oosssokkkdkd}
\end{figure*}

\begin{proof} We will check that there exists an open ball $B_t$ verifying $$B_t\cap G(t)=\emptyset,\mbox{ } B_t\subset C_{{r}}(G(t))$$and
$$\mathbb{P}(B_t\subset C_{{ r}}(\mathcal{X}_n^+(t)), \mbox{ eventually})=1.$$ Since $\mathcal{X}_n^+(t)\subset\mathcal{X}_n^+(\overline{t})$ and, therefore, $ C_{{r}}(\mathcal{X}_n^+(t))\subset C_{{r}}(\mathcal{X}_n^+(\overline{t}))$.\\

Then, let $r^{*}$ be a positive number such that ${r}>r^*>r_0(t)$. Since $G(t)$ is $r_0(t)-$convex, it is verified that $C_{r_0(t)}(G(t))=G(t)\subsetneq C_{r^*}(G(t))$. According to Proposition \ref{Alambda}, there exists $A_t^*= B_{\rho_t^*}(c_t^*)$ satisfying
$$ A_t^*
\subset C_{r^*}(G(t))\subset C_{ r}(G(t)), \mbox{ }A_t^*\cap G(t)=\emptyset \mbox{ and }f(x)>\frac{t}{2},\mbox{ }\forall x\in  A_t^*.\vspace{1.4mm}$$
Since $A_t^{*}\subset C_{r^*}(G(t))$, for all $z\in A_t^{*}$ it is verified that
\begin{equation}\label{de}
	B_{r^*}(w)\cap G(t)\neq \emptyset,\mbox{ }\forall \omega \in B_{r^*}(z).\vspace{1.4mm}
\end{equation}
Let $B_t$ be an open ball of radius $\gamma_t=\rho_t^*/2$ centering at $c_t^*$. Then, $f(x)>l-\zeta$ for all $x\in B_t$  because $B_t\subset A_t^{*}$. In addition, Proposition \ref{Alambda} guarantees that $  A_t^{*}\subset C$ and, therefore, $  B_t\subset C$.\\

Without loss of generality, we assume that ${r}<\frac{\rho_t^*}{2}+r^*$. In fact, if we suppose ${r}\geq \frac{\rho_t^*}{2}+r^*$ then it is possible to consider $r^{**}>r^*$ verifying that ${r}>r^{**}>r_0(t)$ and $r<\frac{\rho_t^*}{2}+r^{**}$. For this $r^{**}$,
$$A_t^{*}\subset C_{r^*}(G(t))\subset  C_{r^{**}}(G(t)) \mbox{ and } A_t^{*}\cap G(t)=\emptyset.\vspace{1.4mm}$$Therefore, it would be enough to consider $r^*=r^{**}$.\\$ $\\
Three steps will be considered for obtaining the first part of the proof.\\$ $\\  	
\underline{Step 1:} If $z\in B_t$ and $\omega\in B_{\frac{\rho_t^*}{2}+r^*}(z)$ then $B_{r^{*}}(\omega)\cap G(t)\neq\emptyset$.\\Then, let $z\in B_t$ and $\omega\in B_{\frac{\rho_t^*}{2}+r^*}(z)$:
\begin{enumerate}
	\item If $\omega\in B_{r^*}(z)$ then, according to equation (\ref{de}), $B_{r^{*}}(\omega)\cap G(t)\neq\emptyset$.
	\item If $\omega\notin B_{r^*}(z)$ then $\omega\in B_{r^*+\frac{\rho_t^*}{2}}(z)\setminus B_{r^*}(z)$. Let $\overline{z}\in [w,z]$
	such that $\|z-\overline{z}\|=\frac{\rho_t^*}{2}$ then $\|\omega-\overline{z}\|<r^*$. Therefore, $\overline{z}\in A_t^{*}\subset C_{r^*}(G(t))$ because $\|\overline{z}-c_t^*\|\leq\|\overline{z}-z\|+\|z-c_t^*\|<\frac{\rho_t^*}{2}+\frac{\rho_t^*}{2}=\rho_t^*$. Since $\overline{z}\in C_{r^*}(G(t))$, according to equation (\ref{de}),
	$$ B_{r^*}(\overline{w})\cap G(t)\neq \emptyset,\mbox{ }\forall \overline{w} \in B_{r^*}(\overline{z}).\vspace{1.4mm}$$Since $\omega \in B_{r^*}(\overline{z})$, $B_{r^*}(w)\cap G(t)\neq \emptyset$.\\$ $\\
\end{enumerate}\underline{Step 2:} If ${r}<r^*+\frac{\rho_t^*}{2}$ and $z\in B_t$ then $B_{r^*}(w)\cap G(t)\neq \emptyset$, $\forall \omega\in B_{{r}}(z)$.\\Of course, if ${r}<r^*+\frac{\rho_t^*}{2}$ and $\omega\in B_r(z)$ then $\omega\in B_{r^*+\frac{\rho_t^*}{2}}(z)$ and, according to Step 1, $B_{r^*}(\omega)\cap G(t)\neq \emptyset$.\\$ $\\  
\underline{Step 3:} If ${r}<r^*+\frac{\rho_t^*}{2}$, with probability one, $$\exists n_0\in \mathbb{N} \mbox{ such that } B_{t}\subset C_{{r}}(\mathcal{X}_n^+(t)), \mbox{ } \forall n\geq n_0.$$Let $z\in B_{t}$.\\According to Step 2,
$$B_{r^*}(\omega)\cap G(t)\neq \emptyset,\mbox{ }\forall \omega \in B_{{r}}(z).\vspace{1.4mm}$$We will prove that, with probability one,
$$\exists n_0\in  \mathbb{N}\mbox{ such that } B_{{r}}(\omega)\cap \mathcal{X}_n^+(t)\neq \emptyset, \mbox{ }\forall \omega \in B_{{r}}(z)\mbox{ and }\forall n\geq n_0.\vspace{1.4mm} $$Let $\omega \in B_{{r}}(z)$ and $s\in B_{r^*}(\omega)\cap G(t)$. According to Corollary  \textcolor{blue}{2.1}, with probability one,
$$\exists n_0\in  \mathbb{N}\mbox{ such that } d_H(\mathcal{X}_n^+(t),G(t))<{r}-r^*, \mbox{ }\forall n\geq n_0.\vspace{1.4mm}$$Then, with probability one and for $n$ large enough,  	
$$\exists x_n\in \mathcal{X}_n^+(t) \mbox{ such that }\|x_n-s\|<{r}-r^*$$and, therefore,  	
$$\|x_n -\omega\|\leq\|x_n-s\|+\|s-\omega\|<{r}-r^*+r^*={r}.\vspace{1.4mm} $$Then,
$$\forall w \in B_r(z),\mbox{ } B_r(w)\cap \mathcal{X}_n^+(t)\neq \emptyset.$$As consequence,
$B_t\subset C_r(\mathcal{X}_n^+(t))$.\\  	
Finally, by construction, $d(z,G(t))\geq 3\rho^*_t$ for all $z\in A_t^*$. Since $B_t\subset A_t^* $ it is verified that $d(z,G(t))\geq 6\gamma_t$, for all $z\in B_t$.\\


If $\gamma_t\leq c/2 $ then the proof is fisnished. If this condition is not satisfied, it is enough to redefine the radius $\gamma_t=c/2$. The new ball $B_t=B_{c/2}(c_t^*)$ would be inside $B_{\rho^*_t/2 }(c_t^*)$. Therefore, it would satisfy all properties required.\qedhere \end{proof}

With probability one and for $n$ large enough, the open ball $B_t$ constructed in Lemma \ref{contenidos} intersects $\mathcal{X}_n^-(\overline{t})$ for $\overline{t}<t$, see Lemma \ref{jijo}.

\begin{lemma}\label{jijo}Let $G(t)$ be a compact, nonempty and nonconvex level set. Under assumptions (A), (D) and (K), let $\mathcal{X}_n$ be a random sample generated by the density function $f$, let $r_0(t) $ be as established in Definition \textcolor{blue}{2.1} and let $\mathcal{X}_n^-(t)$ be as established in Definition \textcolor{blue}{2.2}. Let $B_t$ be the open ball of radius $\gamma_t$ established in Lemma \ref{contenidos}. For $\overline{t}\in[l,u]$ verifying $\overline{t}<t$ and $t-\overline{t}\leq  (mc)/2 $. Then,
$$\mathbb{P}(\mathcal{X}_n^-(\overline{t})\cap B_t\neq \emptyset, \mbox{ eventually})=1.$$See Proposition \ref{enPruebaTeorema3Walther4} for details about constant $c$. \end{lemma}

\begin{proof}It is easy to prove that $\mathbb{P}(\mathcal{X}_n\cap B_t\neq\emptyset, \mbox{ eventually})=1$ using Borel-Cantelli's Lemma. Since	
$$\mathbb{P}(\mathcal{X}_n\cap B_t=\emptyset)=\left[1-\mathbb{P}(X_1\in B_t)\right]^n\leq  e^{-n\mathbb{P}(X_1\in B_t)}$$and $f(x)>l-\zeta>0$ for all $x\in B_t$, it is satisfied that	
\begin{equation}
	\mathbb{P}(X \in B_t)=\int_{B_t}f\mbox{ }d\mu \geq \int_{B_t}l-\zeta\mbox{ }d\mu=(l-\zeta)\mu(B_t)>0
\end{equation}and, as consequence,
$$\sum_{n=1}^\infty \mathbb{P}(\mathcal{X}_n\cap B_t=\emptyset)\leq \sum_{n=1}^\infty e^{-n(l-\zeta)\mu(B_t)}<\infty.$$

Next, an analogous result for $\mathcal{X}_n^-(\overline{t})$ must be obtained. Since $\frac{2}{m}(t-\overline{t})<2\gamma_t$, Proposition \ref{enPruebaTeorema3Walther4} (b) ensures that
\begin{equation}\label{lewl2bb}
	G(\overline{t})\subset G(t)\oplus\frac{2}{m}(t-\overline{t})B_1[0]\subset G(t)\oplus 2\gamma_t B_1[0].
\end{equation}Then,
\begin{equation}\label{lewl2bbc}
	B_t\subset G(\overline{t})^c\ominus 2\gamma_t B_1[0].
\end{equation}Therefore,
$$\forall y\in B_t,\mbox{ }d(y,G(t))\geq 6\gamma_t>0 \mbox{ and }d(y,G(t)\oplus  2\gamma_t B_1[0])\geq  4\gamma_t>0.$$According to equation (\ref{lewl2bb}),
$$\forall y\in B_t,\mbox{ }d(y,G(\overline{t}))\geq 4\gamma_t>2\gamma_t.$$Therefore,
$$ B_t\subset G(\overline{t})^c\ominus 2\gamma_t B_1[0].$$
Since  $0<2\gamma_t\leq c$, Proposition \ref{enPruebaTeorema3Walther4} (a) guarantees that
$$
\sup_{x\in G(\overline{t})^c\ominus 2\gamma_t B_1[0] } f(x)<\overline{t}-b_t\mbox{ for a certain }b_t>0.
$$As consequence,
\begin{equation}\label{hjuyt}
	\sup_{x\in B_t } f(x)<\overline{t}-b_t\mbox{ for a certain }b_t>0.
\end{equation}
The first part of the proof ensures that, with probability one,
\begin{equation}\label{hg}
	\exists n_0\in \mathbb{N}\mbox{ such that }\exists X_i\in\mathcal{X}_n\cap B_t,\mbox{ }\forall n\geq n_0 .
\end{equation}
According to Lemma \ref{contenidos}, $B_t\subset C$ where $C$ is the compact set established in Proposition \ref{enPruebaTeorema3Walther}. Then, $X_i\in C$ and Proposition \ref{enPruebaTeorema3Walther}, guarantees that, with probability one and for $n$ large enough,
$$
\exists N>0 \mbox{ tal que }\sup_{C}|f_n-f|\leq N  \left(  \frac{\log{n}}{n}\right)^{p/(d+2p)} .
$$If $D_n=N  \left(  \frac{\log{n}}{n}\right)^{p/(d+2p)}$ then $\lim_{n\rightarrow \infty}D_n=0$. Therefore, fixed $b_t/2>0$,
$$\exists n_1\in\mathbb{N}\mbox{ such that } D_n<\frac{b_t}{2}, \forall n\geq n_1.$$Since $X_i\in C$, with probability one,
$$|f_n(X_i)-f(X_i)|\leq \sup_{C}|f_n-f|\leq D_n<\frac{b_t}{2},\mbox{ }\forall n\geq\max\{n_0,n_1\}.$$Then, $X_i \in \mathcal{X}_n^-(\overline{t})$ because for all $ n\geq\max\{n_0,n_1\}$,
\[f_n(X_i)\leq f(X_i)+D_n<\overline{t}-b_t+D_n<\overline{t}-b_t+\frac{b_t}{2}=\overline{t}-\frac{b_t}{2}\leq \overline{t}-D_n.\qedhere\]
\end{proof}

Next, it is proved that a ball contained in the the open ball $A_t$ constructed in Proposition \ref{Alambda} is inside the $(r+\epsilon)-$convex hull of $G(\overline{t})$ for fixed values of $\epsilon>0$, $\overline{t}>t$ and $r>r_0(t)$.

\begin{proposition}\label{widetilde}Let $G(t)$ be a compact, nonempty and nonconvex level set. Under assumptions (A), let $r_0(t) $ be as established in Definition \textcolor{blue}{2.1} and $\epsilon>0$. Given $r>0$ verifying $r>r_0(t)$, let $A_t$ be the open ball of radius $\rho_t$ centering at $c_t$ constructed in Proposition \ref{Alambda}.
For $\overline{t}>t$ with $\overline{t}\in[l,u]$ and $\overline{t}-t\leq\min\{\epsilon m/2,mc/2\}$, it is verified that $$D_t=B_{\rho_t/2}(c_t)\subset C_{{r}+\epsilon}(G(\overline{t})).$$See Proposition \ref{enPruebaTeorema3Walther4} for details about constant $c$.

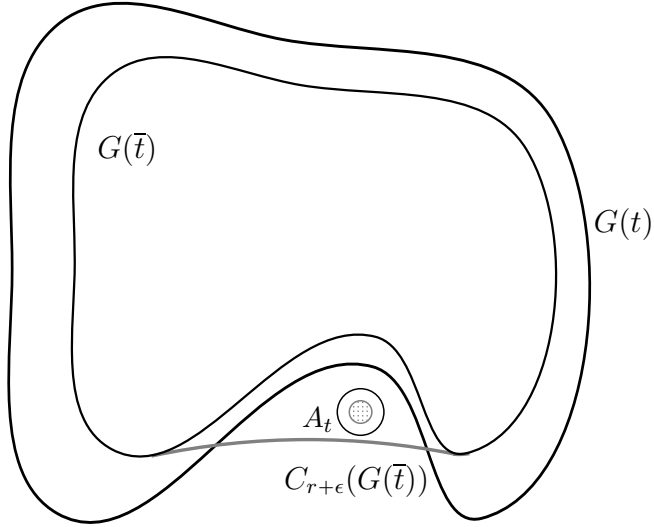
\begin{figure*}[h!]
	\hspace{2cm}\begin{pspicture}(-1.7,-1.9)(10,6.2)
	\psarc[showpoints=false,fillstyle=solid,fillcolor=white,linecolor=gray,linewidth=0.45mm,linearc=3](6.108,.7351) {.8}{264}{280}

	\rput(-1.1,-.85){	\scalebox{1.2}[1.3]{	 \psccurve[showpoints=false,fillstyle=none,fillcolor=white,linecolor=black,linewidth=0.3mm,linearc=3](1.5,0)(1,2.5)(1.5,5)(4,4.85)(7,4)(6.2,-0.05)(5,1.5)}}
	\psarc[showpoints=false,fillstyle=none,fillcolor=white,linecolor=gray,linewidth=0.45mm,linearc=3](4.1,-9.867) {10}{78.9}{102}
	
	\psccurve[showpoints=false,fillstyle=none,fillcolor=white,linecolor=black,linewidth=0.3mm,linearc=3](1.5,0)(1,2.5)(1.5,5)(4,4.85)(7,4)(6.2,-0.05)(5,1.5)
	
	\pscircle[linearc=0.25,linecolor=black,linewidth=0.2mm,linestyle=solid,dash=3pt 2pt](4.8,.5){.32}
	\pscircle[linearc=0.25,linecolor=gray,linewidth=0.2mm,linestyle=solid,dash=3pt 2pt,fillstyle=crosshatch*,fillcolor=gray,hatchcolor=white,hatchwidth=1.2pt,hatchsep=.5pt,hatchangle=0](4.8,.5){.16}
	\rput(1.7,4){$G(\overline{t})$}
	\rput(4.23,0.4){\small{$A_t$}}
	\rput(4.75,-.4){$C_{r+\epsilon}(G(\overline{t}))$}
	\rput(8.3,3){$G(t)$}
	\end{pspicture}\vspace{-.2cm}\caption{Elements of proof in Proposition \ref{widetilde}. $D_t$ in gray color. }\label{oosssokkkdkd}
\end{figure*}

\end{proposition}

\begin{proof}Proposition \ref{Alambda} ensures that there exists an open ball $A_t=B_{\rho_t}(c_t)$ verifying that
$$
A_t\cap B_{2\rho_t}(G(t))=\emptyset,\mbox{ }G(t)\cap A_t=\emptyset\mbox{ and }A_t\subset C_{r}(G(t)).$$Since $\overline{t}>t$, it is verified that $$A_t\cap G(\overline{t})=\emptyset.$$ 	 We can assume, without loss of generality, that $\epsilon\leq \rho_t/2$. In another case, the proof will be done considering $\overline{\epsilon}=\rho_t/2>0$. In this case, it would be verify that $$D_t\subset C_{({r}+\overline{\epsilon})}(G(\overline{t})) \subset C_{({r}+\epsilon)}(G(\overline{t})).$$
According to Proposition \ref{enPruebaTeorema3Walther4} (b), if $ \overline{t}-t\leq\frac{mc}{2}$ then
$$
G(t)\subset G(\overline{t})\oplus\frac{2}{m}(\overline{t}-t)B_1[0].
$$In addition, $G(\overline{t})\subset G(t)$. Therefore,
$$G(\overline{t}) \subset G(t)\subset G(\overline{t})\oplus\frac{2}{m}(\overline{t}-t)B_1[0]$$and, as consequence,
$$d_H(G(t),G(\overline{t}))\leq \frac{2}{m}(\overline{t}-t).$$Since $\overline{t}-t\leq \epsilon m/2$ then
\begin{equation}\label{ioio}
	d_H(G(t),G(\overline{t}))\leq \frac{2}{m}(\overline{t}-t)\leq \epsilon.
\end{equation}
$ $\\
Next, we will prove that $D_t=B_{\rho_t/2}(c_t)\subset C_{(r+\epsilon)}(G(\overline{t}))$.\\Let $x\in D_t\subset C_{r}(G(t))$ and let $B_{r+\epsilon}(c)$ be an arbitrary ball such that $x\in B_{r+\epsilon}(c)$. We will check that $B_{r+\epsilon}(c)\cap G(\overline{t})\neq\emptyset$. Two cases will be distinguished:

\begin{enumerate}
	\item Let assume that $x\in B_{r}(c)$. Since $x\in B_{r}(c)\cap C_{r}(G(t))$, $B_{r}(c)\cap G(t)\neq \emptyset$,
	$$\exists y\in G(t)\cap B_{r}(c).$$According to equation (\ref{ioio}),
	$$\exists z\in G(\overline{t})\cap B_{\epsilon}(y).$$By construction, $z\in B_{{r}+\epsilon}(c)$. Therefore,
	$$B_{{r}+\epsilon}(c)\cap G(\overline{t})\neq\emptyset.$$

	\item Let assume that $x \notin B_{r}(c)$. Taking into account the construction of
	$D_t$, we can guarantee that
	$$D_t\oplus \rho_t/2 B_1[0]\subset A_t \subset C_{r}(G(t)).$$Then, since $x\in D_t$,  $\epsilon\leq \rho_t/2<\rho_t$ and  $d(x,B_{{r}}(c))<\epsilon$,
	$$\exists y\in  B_{{r}}(c)\cap B_{\rho_t/2}(x)\cap C_{{r}}(G(t)).$$Therefore,
	$$\exists y\in  B_{{r}}(c)\cap C_{{r}}(G(t)).$$Then,
	$$\exists z\in  B_{{r}}(c)\cap G(t).$$According to equation (\ref{ioio}),
	$$\exists t \in  B_{{\epsilon}}(z)\cap G(\overline{t}).$$By construction, $t\in B _{{{r}}+{\epsilon}}(c)$. Therefore,
	$$B _{{{r}}+{\epsilon}}(c)\cap G(\overline{t})\neq\emptyset.\qedhere $$

\end{enumerate}

\end{proof}

Similar results to the previous ones are proved in Proposition \ref{bolassss} and Lemma \ref{jijoju}.

\begin{proposition}\label{bolassss}Let $G(t)$ be a compact, nonempty and nonconvex level set. Under assumptions (A), (D) and (K), let $\mathcal{X}_n$ be a random sample generated by the density function $f$, let $r_0(t) $ be as established in Definition \textcolor{blue}{2.1} and let $\mathcal{X}_n^+(t)$ be as established in Definition \textcolor{blue}{2.2}. Let $A_t$ be the open ball of radius $\rho_t$ centering at $c_t$ constructed in Proposici\'on \ref{Alambda}. Let ${ {r}}>0$ verifying ${ {r}}>r_0(t)$ and $\epsilon>0$. It is satisfied that
$$\mathbb{P}(B_{\rho_t/2}(c_t)\subset C_{ {r}+\epsilon}(\mathcal{X}_n^+(t)),\mbox{ eventually})=1.$$
In addition, if $\overline{t}>t$ with $\overline{t}\in[l,u]$ verifying $\overline{t}-t\leq\min\{mc/2,\epsilon m/6\}$ then
$$\mathbb{P}(B_{\rho_t/4}(c_t)\subset C_{ {r}+2\epsilon}(\mathcal{X}_n^+(\overline{t})),\mbox{ eventually})=1.$$See Proposition \ref{enPruebaTeorema3Walther4} for details about constant $c$.

\end{proposition}

\begin{proof}Remember that for all $z\in B_{\gamma_t/2}(c_t)\subset C_{ {r}}(G(t))$ it is verified that
\begin{equation}\label{jujuju}
	B_{ {r}}(\omega)\cap G(t)\neq\emptyset,\mbox{ }\forall\omega\in B_{ {r}}(z).
\end{equation}
It is not restrictive to assume that $ {r}+\epsilon< {r}+ \gamma_t/2$ or, equivalently, $\epsilon<\gamma_t/2$. The proof is established in three steps:\\$ $\\  
\underline{Step 1:} It is necessary to prove that if $z\in B_{\gamma_t/2}(c_t)$ and $\omega\in B_{\frac{\gamma_t}{2}+ {r}}(z)$ then $B_{ {r}}(\omega)\cap G(t)\neq\emptyset$.\\Then, let $z\in B_{\gamma_t/2}(c_t)$ and $\omega\in B_{\frac{\gamma_t}{2}+ {r}}(z)$ :
\begin{enumerate}
	\item If $\omega\in B_{ {r}}(z)$, according to equation (\ref{jujuju}), $B_{ {r}}(\omega)\cap G(t)\neq\emptyset$.
	\item If $\omega\notin B_{ {r}}(z)$ then $\omega\in B_{\frac{\gamma_t}{2}+ {r}}(z)\setminus B_{ {r}}(z)$. Let  $[w,z]$ be the segment of extremes $w$ and $z$ and $ \overline{z}\in [w,z]$ such that $\|\overline{z}- {z}\|=\frac{\gamma_t}{2}$ then $\|\omega- \overline{z}\|< {r}$. Therefore, $ \overline{z}\in B_{\gamma_t}(c_t)\subset C_{ {r}}(G(t))$ since $\| \overline{z}-c\|\leq\| \overline{z}-z\|+\|z-c\|<\frac{\gamma_t}{2}+\frac{\gamma_t}{2}= \gamma_t$. Since $ \overline{z}\in C_{ {r}}(G(t))$, according to equation (\ref{jujuju}),
	$$ B_{ {r}}( {w})\cap G(t)\neq \emptyset,\mbox{ }\forall  {w} \in B_{ {r}}( \overline{z}).\vspace{1.4mm}$$Since $\omega \in B_{ {r}}( \overline{z})$, $B_{ {r}}(w)\cap G(t)\neq \emptyset$.   		
\end{enumerate}$ $\\ \underline{Step 2:} We will prove that if $ {r}+\epsilon< {r}+\frac{\gamma_t}{2}$ and $z\in B_{\frac{\gamma_t}{2}}(c_t)$ then $B_{ {r}}(w)\cap G(t)\neq \emptyset$, $\forall \omega\in B_{ {r}+\epsilon}(z)$.\\If $ {r}+\epsilon< {r}+\frac{\gamma_t}{2}$ and $\omega\in B_{ {r}+\epsilon}(z)$ then $\omega\in B_{ {r}+\gamma_t/2}(z)$ and, according to Step 1, $B_{ {r}}(\omega)\cap G(t)\neq \emptyset$.\\$ $\\ 
\underline{Step 3:} It remains to check that if $ {r}+\epsilon< {r}+\gamma_t/2$ then, with probability one, there exists $n_0\in \mathbb{N}$ such that $ B_{\frac{\gamma_t}{2}}(c_t)\subset C_{ {r}+\epsilon}(\mathcal{X}_n^+(t))$, $\forall n\geq n_0$.\\Then, let $z\in B_{\frac{\gamma_t}{2}}(c_t)$. According to Step 2,  	
$$B_{ {r}}(\omega)\cap G(t)\neq \emptyset,\mbox{ }\forall \omega \in B_{ {r}+\epsilon}(z).\vspace{1.4mm}$$We will prove that, with probability one,
$$\exists n_0\in  \mathbb{N}\mbox{ such that } B_{ {r}+\epsilon}(\omega)\cap \mathcal{X}_n^+(t)\neq \emptyset, \mbox{ }\forall \omega \in B_{ {r}+\epsilon}(z)\mbox{ and }\forall n\geq n_0.\vspace{1.4mm} $$Let $\omega \in B_{ {r}+\epsilon}(z)$ and $s\in B_{ {r}}(\omega)\cap G(t)$. According to Corollary  \textcolor{blue}{2.1}, with probability one,  	
$$\exists n_0\in  \mathbb{N}\mbox{ that does not depend on } s \mbox{ and } z \mbox{ such that }  d_H(\mathcal{X}_n^+(t),G(t))<\epsilon, \mbox{ }\forall n\geq n_0.\vspace{1.4mm}$$Therefore,  	
$$\exists x_n\in \mathcal{X}_n^+(t) \mbox{ such that }\|x_n-s\|<\epsilon, \mbox{ }\forall n\geq n_0.\vspace{1.4mm}$$Then,
$$\|x_n -\omega\|\leq\|x_n-s\|+\|s-\omega\|< {r}+\epsilon= {r}+\epsilon.\vspace{1.4mm} $$Therefore, since $ B_{ {r}+\epsilon}(\omega)\cap\mathcal{X}_n^+(t)\neq \emptyset$ for all $ \omega\in B_{ {r}+\epsilon}(z)$ then  	
\begin{equation}\label{eq1}
	z\in C_{ {r}+\epsilon}(\mathcal{X}_n^+(t)), \mbox{ }\forall n\geq n_0.
\end{equation}
Next, we will prove that, for $0<\overline{t}-t\leq\min\{mc/2,\epsilon m/6\}$, it is verified that
$$\mathbb{P}( B_{\gamma_t/4}(c_t)\subset C_{ {r}+2\epsilon}(\mathcal{X}_n^+(\overline{t})) ,\mbox{ eventually})=1.$$Therefore, the proof of Proposition \ref{bolassss} will be finished. Then, let assume, without loss of generality, that $\epsilon\leq \gamma_t/4$. In another case, we could consider the proof for $\overline{\epsilon}<\gamma_t/4>0$. It would verify that $$B_{\gamma_t/4}(c_t)\subset C_{({r}+2\overline{\epsilon})}(\mathcal{X}_n^+(\overline{t}))\subset C_{({r}+2 \epsilon)}(\mathcal{X}_n^+(\overline{t})).$$
In addition, Proposition \ref{enPruebaTeorema3Walther4} (b) ensures that
$$
G(t)\subset G(\overline{t})\oplus\frac{2}{m}(\overline{t}-t)B_1[0].$$Since $\overline{t}-t\leq \epsilon m /6$,
$$
G(t)\subset G(\overline{t})\oplus\frac{2}{m}(\overline{t}-t)B_1[0]\subset  G(\overline{t})\oplus\frac{\epsilon}{3}B_1[0].$$
Corollary  \textcolor{blue}{2.1} allows us to prove that, with probability one and for $n$ large enough,
$$ d_H( G(t),\mathcal{X}_n^+(t) )\leq \frac{\epsilon}{3} \mbox{ and }d_H( G(\overline{t}),\mathcal{X}_n^+(\overline{t}) )\leq \frac{\epsilon}{3} .$$Using a triangular inequality, with probability one and for $n$ large enough,  	
\begin{equation}\label{ha}
	d_H( \mathcal{X}_n^+(t) ,\mathcal{X}_n^+(\overline{t}) )\leq \epsilon.
\end{equation}
Then, let $z\in B_{\gamma_t/4}(c_t)$ and let $B_{r+2\epsilon}(x)$ an arbitrary ball such that $z\in B_{r+2\epsilon}(x)$. It will be checked that, with probability one and for $n$ large enough, $B_{r+2\epsilon}(x)\cap\mathcal{X}_n^+(\overline{t})\neq \emptyset $. Two cases will be distinguished:\\$ $\\
\underline{Case 1:} $z\in B_{r+\epsilon}(x) $. Taking into account equation (\ref{eq1}), since $z\in B_{\gamma_t/4}(c_t)\subset C_{r+\epsilon}( \mathcal{X}_n^+(t))$, with probability one and for $n$ large enough,
\begin{equation}
	\exists x_n^+\in B_{r+\epsilon}(x)\cap  \mathcal{X}_n^+(t).
\end{equation}Taking (\ref{ha}) into account,
$$\exists \overline{x}_n^+\in \mathcal{X}_n^+(\overline{t}) \mbox{ such that } \|\overline{x}_n^+-x_n^+\|<\epsilon.$$Then, since
$$\|\overline{x}_n^+-x\|\leq\|\overline{x}_n^+-x_n^+\|+\|x_n^+-x\|<\epsilon+r+\epsilon$$it is verified
$$\overline{x}_n^+\in B_{r+2\epsilon}(x) \cap\mathcal{X}_n^+(\overline{t}).$$\\
\underline{Case 2:} $z\notin B_{r+\epsilon}(x) $. Then,
$$B_{\gamma_t/4}(c_t) \subset  B_{\gamma_t/2}(c_t)\subset C_{r+\epsilon}(\mathcal{X}_n^+(t)).$$Remember that $z\in B_{\gamma_t/4}(c_t)$ and $d(z,B_{r+\epsilon}(x))<\epsilon<\gamma_t/4$. Therefore,
$$\exists y\in B_{\gamma_t/4}(z)\cap C_{r+\epsilon}(\mathcal{X}_n^+(t))\cap B_{r+\epsilon} (x).$$Therefore, with probability one and for $n$ large enough,
$$\exists x_n^+\in \mathcal{X}_n^+(t)\cap B_{r+\epsilon} (x).$$
Taking (\ref{ha}) into account,
$$\exists \overline{x}_n^+\in \mathcal{X}_n^+(\overline{t}) \mbox{ such that } \|\overline{x}_n^+-x_n^+\|<\epsilon.$$It is easy to prove that $\overline{x}_n^+\in  B_{r+2\epsilon} (x)$.\qedhere
\end{proof}

\begin{lemma}\label{jijoju}	Let $G(t)$ be a compact, nonempty and nonconvex level set. Under assumptions (A), (D) and (K), let $\mathcal{X}_n$ be a random sample generated by the density function $f$ and let $\mathcal{X}_n^-(t)$ be as established in Definition \textcolor{blue}{2.2}. Let $A_t$ be the open ball of radius $\rho_t$ centering at $c_t$ constructed in Proposition \ref{Alambda}. Let $\overline{t}>t$ with $\overline{t}\in[l,u]$, $a>1$ and $a\in\mathbb{N}$ fixed. Then,
$$\mathbb{P}(\mathcal{X}_n^-(\overline{t})\cap B_{\rho_t/a}(c_t) \neq \emptyset,\mbox{ eventually})=1.$$\end{lemma}
\begin{proof} We will establish the proof in three steps:

$ $\\
\underline{Step 1:} Given $a>1$, we need to check that $$\mathbb{P}(\mathcal{X}_n\cap  B_{\rho_t/a}(c_t)\neq\emptyset, \mbox{ eventually})=1.$$See proof of Lemma \ref{jijo}, it is totally analogous. Therefore, with probability one,
\begin{equation}\label{hgg}
	\exists n_0\in \mathbb{N}\mbox{ such that }\forall n\geq n_0 \mbox{ it is verified that }\mathcal{X}_n\cap B_{\rho_t/a}(c_t)\neq\emptyset.
\end{equation}

$ $\\
\underline{Step 2:} According to Step 1, if $x\in A_t$ then $f(x)<t$ since $A_t\cap G(t)=\emptyset$. Since $\overline{t}>t$ then
\begin{equation}\label{hjuytj}
	f(x)<t<\overline{t} -(\overline{t}-t)/2\mbox{ for all }x\in A_t.
\end{equation}
According to equation (\ref{hgg}),
$$\exists X_i\in\mathcal{X}_n\cap B_{\rho_t/a}(c_t),\mbox{ }\forall n\geq n_0, .$$
In addition, it is verified that $ B_{\rho_t/a}(c_t)\subset A_t\subset C$, where $C$ is a compact set established in Proposition \ref{enPruebaTeorema3Walther} that ensured that
$$\sup_{C}|f_n-f|=O\left(  \left(  \frac{\log{n}}{n}\right)^{p/(d+2p)}\right).$$Therefore, and for $n$ large enough,
$$
\exists N>0 \mbox{ such that }\sup_{C}|f_n-f|\leq N  \left(  \frac{\log{n}}{n}\right)^{p/(d+2p)} .
$$
If $D_n=N  \left(  \frac{\log{n}}{n}\right)^{p/(d+2p)}$ then $\lim_{n\rightarrow \infty}D_n=0$. So, for $(\overline{t}-t)/4>0$,
$$\exists n_1\in\mathbb{N}\mbox{ tal que } D_n<\frac{\overline{t}-t}{4}, \forall n\geq n_1.$$Since $X_i\in B_{\rho_t/a}(c_t)\subset C$, with probability one,
$$|f_n(X_i)-f(X_i)|\leq \sup_{C}|f_n-f|\leq D_n<\frac{\overline{t}-t}{4},\mbox{ }\forall n\geq\{n_0,n_1\}.$$Taking (\ref{hjuytj}) into account, for $n$ large,
$$f_n(X_i)\leq f(X_i)+D_n< f(X_i)+\frac{\overline{t}-t}{4}<\overline{t} -\frac{\overline{t}-t}{2}+\frac{\overline{t}-t}{4}=\overline{t}-\frac{\overline{t}-t}{4}<\overline{t}-D_n.$$
As consequence, for $n$ large enough,
\[\exists X_i\in\mathcal{X}_n\cap B_{\rho_t/a}(c_t)\mbox{ such that }  X_i\in\mathcal{X}_n^-(\overline{t}).\qedhere\]
\end{proof}

\end{document}